\documentclass[11pt,a4paper]{article}
\usepackage[OT1]{fontenc}
\usepackage[english]{babel}
\usepackage{amsfonts}
\usepackage{amsthm}
\usepackage{amsmath}

\def\h{ {\cal H} }
\def\z{ {\cal Z} }
\def\l{ {\cal L} }

\def\b{ {\cal B} }
\def\u{ {\cal U} }

\def\m{ {\cal M} }

\def\t{ {\cal T} }
\def\s{ {\cal S} }
\def\e{ {\cal E} }
\def\p{ {\cal P} }

\def\k{ {\cal K} }

\def\c{ {\cal C} }

\def\j{ {\cal J} }

\def\T{{\mathbb{T}}}
\def\CC{ {\mathbb{C}} }
\def\RR{ {\mathbb{R}} }
\def\Z{\mathbb{Z}}

\def\la{\lambda}
\def\fde{{\bf \Delta}}
\def\fga{{\bf \Gamma}}

\def\corank{\mathrm{corank}}

\def\BB{ {\bf B} }

\def\noi{\noindent}

\def\D{\mathbb{D}}

\newtheorem{teo}{Theorem}[section]
\newtheorem{prop}[teo]{Proposition}
\newtheorem{lem}[teo]{Lemma}
\newtheorem{coro}[teo]{Corollary}

\theoremstyle{definition}
\newtheorem{rem}[teo]{Remark}
\newtheorem{ejem}[teo]{Example}

\title{Subspaces with or without  a common complement}
\author{Esteban Andruchow and Eduardo Chiumiento}

\begin{document}

\maketitle 

\begin{abstract}
Let $\h$ be a separable complex Hilbert space. Denote by $Gr(\h)$ the Grassmann manifold of $\h$.  We study the following sets of pairs of  elements in $Gr(\h)$:
$$
{\bf \Delta}=\{(\s,\t)\in Gr(\h)\times Gr(\h): \exists \z \in Gr(\h) \hbox{ such that } \s\dot{+}\z=\t\dot{+}\z=\h \},
$$
which are pairs of subspaces that have a common complement, and 
$$
\fga=Gr(\h) \times Gr(\h) \setminus \fde,
$$
which are pairs of subspaces that do not admit a common complement. 
We identify $\s\sim P_\s$, the subspace $\s$ with the orthogonal projection $P_\s$ onto $\s$. Thus we may regard ${\bf \Delta}$ and ${\bf \Gamma}$ as subsets of $\b(\h)\times\b(\h)$ (here $\b(\h)$ denotes the algebra of bounded linear operators in $\h$). We show that ${\bf \Delta}$ is open,  and its connected components are parametrized by the dimension and codimension of the subspaces. The connected component of $\fde$ having both infinite dimensional and co-dimensional subspaces is dense in the corresponding component of $Gr(\h)\times Gr(\h)$.  On the other hand,  ${\bf \Gamma}$ is a (closed) $C^\infty$ submanifold of $\b(\h)\times\b(\h)$, and we characterize the connected components of ${\bf \Gamma}$ in terms of dimensions and semi-Fredholm indices. We study the role played by the geodesic structure of  the Grassmann geometry of $\h$ in the geometry of both ${\bf \Delta}$ and ${\bf \Gamma}$. Several examples of pairs in ${\bf \Delta}$ and the connected components of ${\bf \Gamma}$ are given in Hilbert spaces of functions.
\end{abstract}

\bigskip



\section{Introduction}

These notes reflect on the paper \cite{lauzontreil} by M. Lauzon and S. Treil. These authors characterize pairs of (closed) subspaces $\s$ and $\t$ of a  Hilbert space $\h$, which have a common complement $\z$. That is, there exists a closed subspace $\z\subset \h$ such that 
$$
\s\dot{+}\z=\t\dot{+}\z=\h,
$$
where the symbol $\dot{+}$ stands for direct sum. In this paper we investigate  geometric aspects of the set of all the pairs of subspaces having a common complement, and the set of all pairs that do not admit a common complement,

Denote by $\b(\h)$ the algebra of bounded linear operators in a complex separable Hilbert space $\h$,   $\p(\h)$ the subset of orthogonal projections, and $Gr(\h)$ the Grassmann manifold of $\h$.
Throughout, we identify subspaces in $Gr(\h)$  with their corresponding orthogonal projections in $\p(\h)$.
%
Our objects of study will be the complementary sets
\begin{equation*}
\fde:=\{(P_\s,P_\t)\in\p(\h) \times\p(\h): \s \hbox{ and } \t \hbox{ have a common complement}\},
\end{equation*}
and 
\begin{equation*}
\fga:=\{(P_\s,P_\t)\in\p(\h)\times\p(\h): \s \hbox{ and } \t \hbox{ do not have a common complement}\}.
\end{equation*}
The space $\p(\h)$  is a complemented C$^\infty$ submanifold of $\b(\h)$ (see \cite{cpr, pr}). Therefore it is natural to ask about the geometric structure of the sets $\fde$ and $\fga$. By an elementary argument, or using a result by J. Giol \cite{giol}, it can be shown that $\fde$ is an open subset of $\p(\h)\times\p(\h)$. Its complement ${\fga}$ is shown to be a (non complemented) closed $C^\infty$ submanifold of $\p(\h)\times\p(\h)$.  
Thus, both spaces have differentiable structure. We study how the geometry of these spaces relates  to geometry of $\p(\h)$. Specifically, concerning the geodesics of $\p(\h)$, which have been thoroughly studied. We are also interested in the study of concrete examples of pairs of subspaces in $\fde$ and $\fga$ in the context of Hilbert spaces of functions.

Let us describe the contents of this paper. In Section \ref{section 2} we recall preliminaries on common complemented subspaces obtained by M. Lauzon and S. Treil \cite{lauzontreil},  and also the results that were later proved by J. Giol \cite{giol}. We also state basic facts  on  the following aspects of projections: the geometry of $\p(\h)$ \cite{cpr, pr}, Halmos' model for a pair of subspaces \cite{halmos}, and the index of a pair of projections \cite{ass}.

 In Section \ref{section 3} we study the structure of the set ${\bf \Delta}$, and show that ${\bf \Delta}$ is an open subset of $\p(\h) \times \p(\h)$. This is a consequence of a result by Giol \cite{giol}; we give though an elementary proof of this fact. We characterize the connected components of $\fde$ in terms of ranks and co-ranks, and the there is only one connected component $\fde_\infty$ having subspaces of both infinite dimension and codimension. 
 We show by using Halmos' model that  $\fde_\infty$ is dense in the corresponding connected component of the (product) Grassmann manifold. 

In Section \ref{without com compl} we consider the structure of the set ${\bf \Gamma}$. The main result states that ${\bf \Gamma}$ is a (closed) non complemented $C^\infty$ submanifold of $\b(\h)\times\b(\h)$. For  $1\le i,j\le +\infty$, denote by $\p_{i,j}$ the connected  component of $\p(\h)$ consisting of projections with rank $i$ and co-rank $j$. We show that ${\bf \Gamma}$ consists of the following parts:
\begin{itemize}
\item
$\p_{i,j}\times\p_{k,l}$, for $i\ne k$ or $j\ne l$.
\item
For pairs of projections  in $\p_{\infty,\infty}$,
$$
{\bf \Gamma}_1=\{(P_\s,P_\t)\in{\bf \Gamma}: \dim \s\cap\t^\perp<\infty, \   \dim \s^\perp\cap\t<\infty\}.
$$
and
$$
{\bf \Gamma}_\infty=\{(P_\s,P_\t)\in\Gamma: \hbox{ only one of } \s\cap\t^\perp, \   \s^\perp\cap\t \hbox{ has infinite dimension}
\}.
$$
\end{itemize}

Moreover, ${\bf \Gamma}_1$ splits in the following submanifolds, which are its connected components paramatrized by the integers as follows:
$$
{\bf \Gamma}_1^n=\{(P_\s,P_\t)\in{\bf\Gamma}_1: P_\s\big|_\t:\t\to\s \hbox{ is a Fredholm operator of index } n \}, \ n\in\mathbb{Z}. 
$$
The connected components of the submanifold ${\bf \Gamma}_\infty$ are
$$
{\bf\Gamma}_\infty^l:=\{(P_\s,P_\t)\in\p(\h)\times\p(\h): \dim \s\cap\t^\perp<\infty, \dim\s^\perp\cap\t=+\infty\}
$$
and 
$$
{\bf\Gamma}_\infty^r:=\{(P_\s,P_\t)\in\p(\h)\times\p(\h): \dim \s\cap\t^\perp=\infty, \dim\s^\perp\cap\t<+\infty\}.
$$
These results are based on previous results obtained in \cite{pq compacto}.
 
In Section \ref{geodesicas en delta y gamma} we discuss when geodesics of $\p(\h)\times\p(\h)$ remain inside  $\fde$ or $\fga$ along their paths. We give an example which shows that arbitrary close points in $\fde$ may not be joined by a minimizing geodesic of  $\p(\h)\times\p(\h)$. On the other hand,  cases where  geodesics remain in $\fde$ or $\fga$  can be are obtained by imposing conditions related to the ranges or nullspaces, or by taking special kinds of projections such as the ones defined by the restricted Grassmannian (see \cite{carey, PS, SW}).

In Section \ref{examples section} we present several examples of subspaces with or without a common complement related to Hilbert spaces of functions. The existence of geodesics between shift-invariant subspaces of the usual Lebesgue space $L^2(\mathbb{T})$ ($\mathbb{T}$ the unit circle) has been related 
to the injectivity problem for Toeplitz operators (see \cite{grassH2, MP10, polto}). We now examine when several of those examples discussed  in relation to the existence of geodesics   satisfy the more general condition of having a common complement. Also  we consider other examples related to Blaschke products in the Hardy space of the unit disk \cite{ACV21}  or the uncertainty principle \cite{AC19}.  In the case of pairs of subspaces without a common complement, we further determine to which of the above described component of $\fga$ the pair belongs.

\section{ Basic facts}\label{section 2}

 Let $\h$ be a separable complex Hilbert space.  We denote by $\b(\h)$ the algebra of bounded operators acting on $\h$.     For $T \in \b(\h)$, let $R(T)$ and $N(T)$ be the range and nullspace of $T$, respectively. 
 The symbol $\dot{+}$ stands for direct sum of subspaces of $\h$; while we reserve the symbol $\oplus$ for orthogonal sums of subspaces. We write $P_\s\in \b(\h)$ for the  orthogonal projection with range $\s$. Given a  decomposition $\h=\s\dot{+}\z$, we will denote by $P_{\s\parallel\z} \in \b(\h)$ the idempotent with range $\t$ and nullspace $\z$. 
Let us state below several preliminary facts. 



\medskip

\noi \textit{Subspaces with a common complement.} Two closed subspaces $\s, \t \subset \h$ have a \textit{common complement} if there exists a closed subspace $\z\subset \h$ such that $\s\dot{+}\z=\t\dot{+}\z=\h.$

\begin{rem}\label{choreo} 
Let us state the following results from \cite{lauzontreil} by M. Lauzon and S. Treil.

\medskip

\noi $1.$  The subspaces $\s,\t$ have a common complement if and only if there exist a bounded (not necessarily orthogonal) projection $E$ onto one of the subspaces (say $\t$), such  that 
$$
E\big|_\s:\s\to\t
$$
is an isomorphism  (\cite[Prop. 1.3]{lauzontreil}). This result is valid even in the context of Banach spaces.

\medskip

\noi $2.$ Denote by $G=P_\t\big|_\s$, regarded as an operator $G:\s\to\t$ (so that $G^*:\t\to\s$ is $P_\s\big|_\t$). Denote by $\e$ the projection valued spectral measure of $G^*G$.
Note that 
$$
N(G)=\s\cap\t^\perp \ \hbox{ and } \ N(G^*)=\s^\perp\cap\t.
$$
Then $\s$ and $\t$ have a common complement if and only if
\begin{equation}\label{la condicion}
\dim N(G) +\dim \e(0,1-\epsilon)\s=\dim N(G^*)+ \dim \e(0,1-\epsilon)\s
\end{equation}
for some $\epsilon>0$ (equivalently, for all sufficiently small $\epsilon>0$). This characterization also holds  for non separable Hilbert spaces  (\cite[Thm. 0.1]{lauzontreil}).

\medskip

\noi $3.$  As a straightforward consequence of the previous statement, Lauzon and Treil oberved  that  the subspaces $\s$ and $\t$ do not have a common complement in a separable Hilbert space $\h$ if and only if $\dim \s\cap\t^\perp\ne \dim\s^\perp\cap\t$ and the operator $(1-G^*G)\big|_{N(G)^\perp}$ is compact 
 (\cite[Rem. 0.5]{lauzontreil}).
\end{rem}

\begin{rem}\label{Charact Giol}
Later on J. Giol proved the following equivalences (see \cite[Prop. 6.2]{giol}): 
\begin{enumerate}
\item[i)] $\s$ and $\t$ are subspaces with a common complement. 
\item[ii)] There exists $P\in \p(\h)$ such that  $\| P_\s - P \| <1$ and $\| P - P_\t\|<1$. 
\item[iii)]   There exist four idempotents $P_1,P_1',P_2,P_2'$ such that $\s=R(P_1)$, $R(P_1')=R(P_2')$, $R(P_2)=\t$ and $\s^\perp=N(P_1')$, $N(P_1)=N(P_2)$, $N(P_2')=\t^\perp$.
\end{enumerate}
 \end{rem}

\medskip

\noi \textit{The Grassman manifold.} The Grassmann manifold $Gr(\h)$ of $\h$ is defined as the set of all the closed subspaces of $\h$. We identify the Grassmann manifold with the manifold of all orthogonal projections in $\h$ given by 
$$
\p(\h)=\{P \in \b(\h) : P=P^2=P^*\}.
$$
 We observe that this identification can be made more precise: there is a diffeomorpshim between the presentation as subspaces and orthogonal projections (see \cite{AM07}). $\p(\h)$ is a complemented submanifold of $\b(\h)$ with tangent space $(T\p(\h))_P$ at $P\in \p(\h)$ given by
 $$
 (T\p(\h))_P=\{   XP - PX : X=-X^*\in \b(\h)\}. 
  $$
Notice that tangent vectors are co-digonal selfadjoint operators. For each projection $P$, one can define the idempotent $\e_P: \b(\h) \to (T\p)_P$, $\e_P(X)=PX(1-P)+(1-P)XP$. The family $\{ \e_P\}_{P \in \p(\h)}$ induces a linear connection in $\p(\h)$: for $X(t)$ a tangent vector field along a curve $\gamma(t)\in \p(\h)$, set
$$
\frac{DX}{dt}=\e_{\gamma}(\dot{X}).
$$
Let $Gl(\h)$ be the group of all the invertible operators, and $\u(\h)$ the subgroup consisting of unitary operators. Recall that $\u(\h)$ acts on $\p(\h)$ by $U\cdot P_{\s}=P_{U(\s)}=UP_\s U^*$,  $U\in \u(\h)$.  

The geodesics of this connection can be computed (see \cite{pr}, \cite{cpr}): they are induced by one parameter unitary groups, namely, the geodesics starting a $\s$ are of the form:
$$
\delta(t)=e^{itZ} P_\s e^{-itZ},
$$
where $Z^*=Z$ is  co-diagonal with respect to $\s$, meaning that $Z(\s)\subset\s^\perp$ and $Z(\s^\perp)\subset\s$ (i.e., $Z$ has co-diagonal matrix in terms of the decomposition $\h=\s\oplus\s^\perp$). It is useful to pick also $\|Z\|\le \pi/2$, we shall call these  normalized geodesics. It is known (see  \cite{cpr, pr} or also \cite{p-q}) that there exists a geodesic between $\s$ and $\t$ if and only if
\begin{equation}\label{existencia}
\dim \s\cap\t^\perp=\dim \s^\perp\cap\t.
\end{equation}
Moreover, there exists a  unique normalized geodesic joining $\s$ and $\t$ if and only if  $\s\cap\t^\perp=\{0\}= \s^\perp\cap\t$.
 
We shall measure the distance between closed subspaces $\s,\t\subset\h$ by means of the corresponding orthogonal projections and the usual operator norm: 
$$
d(\s,\t):=\|P_\s-P_\t\|.
$$
If $\|P_\s-P_\t\|<1$, then one can find a unitary operator $U=U_{\s,\t}$ which is a $C^\infty$ expression in terms of $P_\s,P_\t$ such that $U\s=\t$. There is more than one way to do this, but if one chooses the geometric argument developed in \cite{pr} and \cite{cpr}, based on the existence of a unique minimal geodesic of the Grassmann manifold, one has
$$
\|U_{\s,\t}-1\|=2 \sin\left( \frac {\arcsin(\|P_\s-P_\t\|)}{2}\right).
$$

\medskip
We remark that the connected components of $\p(\h)$ are parametrized by the rank and the co-rank.
We denote by $\p_{i,j}$ the connected component of $\p(\h)$ consisting of projections with rank $i$ and corank $j$, where the indices satisfy $0 \leq i,j \leq \infty$ and $i + j =\infty$ (usual convention if both are infinite).

\medskip

\noi \textit{Halmos' model for a pair of projections / subspaces.}
Several papers consider the following decomposition of a Hilbert space in the presence of two projections (or subspaces): mainly Dixmier \cite{dixmier}, Halmos \cite{halmos}. We transcribe Halmos' model. Let 
$P$, $Q$ be ortohogonal projections in $\h$, then the following five space decomposition reduces $P$ and $Q$:
\begin{equation}\label{paul h}
\h=R(P)\cap R(Q)\oplus N(P)\cap N(Q)\oplus R(P)\cap N(Q)\oplus N(P)\cap R(Q)\oplus \h_0,
\end{equation}
where $\h_0$ is the orthogonal complement of the sum of the former four, and is called the generic part of $P$ and $Q$. 
The subspace $\h_0$ reduces both $P$ and $Q$, we shall denote the reductions by $P_0:=P\big|_{\h_0}$, $Q_0:=Q\big|_{\h_0}$.

There exists a unitary isomorphism between $\h_0$ and a product space $\l\times\l$, and a positive operator $X$ in $\l$ with $N(X)=\{0\}$ and $\|X\|\le\pi/2$, such that the projections $P$ and $Q$ correspond, under this isomorphism, with the matrices (in $\l\times\l$)
$$
P_0\sim \left(\begin{array}{cc} 1 & 0 \\ 0 & 0 \end{array}\right) \ \hbox{ and } \ Q_0\sim \left(\begin{array}{cc} C^2 & CS \\ CS & S^2 \end{array}\right).
$$
Here $C=\cos(X)$ and $S=\sin(X)$.

\begin{rem}\label{Halmos Giol}
The subspace $\l$ is determined  by the following decomposition  
$$
R(P)=R(P) \cap R(Q) \oplus R(P) \cap N(Q) \oplus \l .
$$ 
Take $\s=R(P)$, $\t=R(Q)$ in Remark \ref{choreo}, then $N(G)^\perp=\s \cap \t \oplus \l$ and the operator $(1 -G^*G)|_{N(G)^\perp}$ is compact if and only if $S^2=\sin(X)$ is compact in $\l$. Equivalently, $X$ is compact in $\l$. Thus, $(\s, \t) \in \fde$ if and only if $\dim \s^\perp \cap \t= \dim \s \cap \t^\perp$ or $X$ is not compact.
This characterization was observed in \cite[Rem. 4.2]{giol}. 
\end{rem}

\medskip

\noi \textit{Index of a pair of projections.}
The following facts were taken from \cite{ass}.
A pair $(P,Q)$ of orthogonal projections is called a \textit{Fredholm pair} if 
$$
Q\Big|_{R(P)}:R(P)\to R(Q)
$$ 
is a Fredholm operator. Its index is
$$
{\rm index}(P,Q)={\rm index}\left(Q\Big|_{R(P)}:R(P)\to R(Q)\right)=\dim R(P)\cap N(Q)-\dim N(P)\cap R(Q).
$$
Among the various characterizations of Fredholm pairs we transcribe the following. 
\begin{rem}\label{remark para Gamma1}
$(P,Q)$ is a Fredholm pair if and only if 
\begin{enumerate}
\item 
$+1$ and $-1$ are isolated in the spectrum $\sigma(P-Q)$ of $P-Q$.
\item 
$R(P)\cap N(Q)$ and $N(P)\cap R(Q)$ are finite dimensional.
\end{enumerate}
\end{rem}
 Notice that a condition for a pair of projections $(P,Q)$  that is stronger than being a Fredholm pair, which shall appear later on, is that $P-Q$ is compact. It is an easy exercise that this is in fact so.

\section{Structure of the set of pairs of subspaces with a common complement}\label{section 3}

In this section we discuss elementary properties  of $\fde$. 
We begin by determining its  connected components.

\begin{rem}\label{comp fin}
Recall that $\p_{i,j}$ denotes de connected components of $\p(\h)$. If $i=k<\infty$ or $j=l<\infty$, then $\p_{i,j} \times \p_{k,l} \subseteq \fde$. Indeed, take $(P,Q) \in \p_{i,j} \times \p_{k,l}$ with $i=k$. Since $T:=Q|_{R(P)}: R(P) \to R(Q)$ is an operator defined in finite-dimensional spaces, we have $k=\dim N(T) + \dim R(T)=\dim N(T^*) + \dim R(T^*)$. From 
$\dim R(T)=\dim N(T)^\perp=\dim R(T^*)$, it follows that $\dim R(P) \cap N(Q)=\dim N(T)=\dim N(T^*)= \dim R(Q) \cap N(P)$.  Hence  $(P,Q) \in \fde$.  Next we observe that the case  where $j=l < \infty$ follows by using the previous case and the map $\perp : \p(\h) \times \p(\h) \to \p(\h) \times \p(\h)$, $\perp(P,Q)=(P^\perp, Q^\perp)$. This is an isomorphism that preserves connected components, and from  the characterization in Remark \ref{Charact Giol}, one easily sees that $\perp(\fde)=\fde$.

On the other hand, assume now that  $i\neq k$ or $j\neq l$, and take $(P,Q) \in \p_{i,j} \times \p_{k,l}$ with $i\neq k$. Thus, $R(P)$ nad $R(Q)$ cannot be isomorphic,    and according to Remark \ref{choreo}  $i)$, we get $(P,Q) \notin \fde$. One can consider also the case where $(P,Q) \in \p_{i,j} \times \p_{k,l}$ with $j\neq l$, which follows again that  $(P,Q) \notin \fde$ by using the map $\perp$ as above.

From these facts, we obtain that 
$$
\fde_{ij}:=\fde \cap (\p_{i,j} \times \p_{i,j})=\p_{i,j} \times \p_{i,j}, 
$$
whenever $i<\infty$  or $j <\infty$, are the only connected components of $\fde$ with finite dimensional rank or corank.
\end{rem}

To analize the case of $\p_{\infty, \infty}$ we need the following (known) property:

\begin{lem}\label{lema menor que 1}
Let $P,Q$ be orthogonal projections such that $\|P-Q\|<1$, and let $P(t)$ be the unique minimal geodesic of $\p(\h)$ such that $P(0)=P$ and $P(1)=Q$. Then for all $t\in[0,1]$ we have that $\|P-P(t)\|<1$.
\end{lem}
\begin{proof}
$P(t)$ is of the form $P(t)=e^{itX}Pe^{-itX}$, for $X^*=X$, $P$-co-diagonal with $\|X\|<\pi/2$ (see Section \ref{section 2}). That $X$ is $P$-co-diagonal means that $X$ anti-commutes with $2P-1$. Then 
$$
2P(t)-1=e^{itX}(2P-1)e^{-itX}=e^{i2tX}(2P-1),
$$
and thus 
$$
\|P-P(t)\|=\frac12\|(2P-1)-(2P(t)-1)\|=\|(e^{i2tX}-1)(2P-1)\|=\|e^{i2tX}-1\|<1,
$$
because $\|2tX\|\le 2\|X\|<\pi$, by a standard functional calculus argument. 
\end{proof} 
\begin{teo}
The subset $\fde_\infty$ of $\fde$, consisting of pairs of pojections in $\p_{\infty, \infty}$ with a common complement, is arcwise connected. Therefore, in view of the above remark, $\fde_\infty$ is the connected component of such pairs.
\end{teo}
\begin{proof}
Let $(P_\s,P_\t)\in\fde_\infty$. We shall prove that there is a continuous path inside $\fde$ connecting $(P_\s,P_\t)$ with $(P_\s, P_\s)$. We proceed in two steps. First we show that there is a continuous path inside $\fde$ connecting $(P_\s, P_\t)$ with a pair $(P_\s, E)$ such that $\|P_\s-E\|<1$. Indeed, by Remark \ref{Charact Giol}  there exists $E\in\p(\h)$ such that $\|P_\s-E\|<1$ and $\|P_\t-E\|<1$. Let $E(t)$ be the minimal geodesic of $\p(\h)$ with $E(0)=E$ and $E(1)=P_\t$. Then the curve $(P_\s, E(t))$ remains inside $\fde$ for $t\in[0,1]$. This follows   again using the result by Giol, for we have the intermediate projection $E$ satisfying $\|P_\s-E\|<1$ and $\|E(t)-E\|<1$ (by Lemma \ref{lema menor que 1}). Next, we find a continuous path inside $\fde$ connecting $(P_\s, E)$ with $(P_\s, P_\s)$. Let $P(t)$ be the minimal geodesic joining $P(0)=P_\s$ and $P(1)=E$. Then the curve $(P_\s, P(t))$ remains inside  $\fde$ for $t\in[0,1]$, since by Lemma \ref{lema menor que 1} we know that $\|P_\s-P(t)\|\le\|P-E\|<1$. 

The proof of the theorem follows showing that any two pairs $(P_\s,P_\s)$ and $(P_{\s'}, P_{\s'})$ with $\s,\s'$ infinite and co-infinite, can be joined by a continuous path inside $\fde$. To this effect, note that if $\s$ and $\t$ have a common complement $\z$ and $U$ is a unitary operator, then $U\s$ and $U\t$ also have a common complement (namely $U\z$), i.e., $(P_\s,P_\t)\in\fde$ implies that $(UP_\s U^*,UP_\t U^*)=(P_{U\s},P_{U\t})\in\fde$.
Then, since $P_\s, P_{\s'}\in\p_{\infty, \infty}$, there exists a continuous path of unitaries $U(t)$ such that $U(0)=1$ and $U(1)\s=\s'$. Then $(U(t)P_\s U^*(t),U(t)P_\s U^*(t))$ is a continuous curve in $\fde$ wich joins $(P_\s,P_\s)$ and $(P_{\s'}, P_{\s'})$. 
\end{proof}

\begin{teo}\label{open delta}
 The set ${\bf \Delta}$ is open in $\p(\h)\times\p(\h)$, and consequently, $\fde$ is a submanifold of $\p(\h) \times \p(\h)$. 
\end{teo} 
\begin{proof}
The statement follows immediately from Giol's characterization in Remark \ref{Charact Giol} $ii)$, which is clearly an open condition. We also give here an elementary proof based on the first item of Remark \ref{choreo}. Pick $(P_\s,P_\t)\in{\bf \Delta}$. Let $\s'$ be close to $\s$ and $\t'$  close to $\t$. Suppose that $\|P_{\s'}-P_\s\|<1$  and $\|P_{\t'}-P_\t\|<1$. Then there exists a unitary operator $U=U_{\s,\s'}$ (a continuous explicit expression in terms of $P_\s$ and $P_{\s'}$) such that $U\s=\s'$, and $\|U-1\|$ is controlled by $\|P_{\s'}-P_\s\|$. Similarly, there exists a unitary operator $V$ with analogous properties for $\t$ and $\t'$ (e.g., $V\t=\t'$, etc.). Then there exists a projection $E$ onto $\t$ such that $E\big|_\s:\s\to\t$ is an isomorphism. We fix $E$, and we claim that for $\t'$ sufficiently close to $\t$, $EV^*\big|_\s:\s  \to \t$ is also an isomorphism. Indeed, note that 
$$
\|EV^*-E\|=\|E(V^*-1)\|\le \|E\| \|V-1\|.
$$
In particular,  we can restrict both $EV^*$ and $E$ to $\s$, and regard these operators as elements in $\b(\s,\t)$. Since $E$ is an isomorphism by the open mapping theorem there exists a constant $r_E=r_{E,\s,\t}$ (depending  on  $E$ fixed, $\s$ and $\t$) such that for any  $T\in\b(\s,\t)$, $\|T-E\|<r_E$ implies that $T$ is an isomorphism. It follows that if $\t'$ is close enough to $\t$ so that $\|V-1\|<\frac{r_E}{\|E\|}$, then $EV^*\big|_\s:\s\to\t$ is an isomorphism. Since $V$ maps $\t$ onto $\t'$,  $VEV^*\big|_\s:\s\to\t'$ is also an isomorphism. Note that $VEV^*$ is a (not necessarily orthogonal) projection onto $\t'$. We claim that $VEV^*\big|_{\s'}:\s'\to\t'$ is also an isomorphism for $\s'$ close enough to $\s$.
Since $U\s=\s'$ and $V\t=\t'$, this is equivalent to  saying that 
$$
EV^*\big|_{\s'}U^*\big|_\s=EV^*U^*\big|_\s:\s\to\t
$$
is an isomorphism. Note that
$$
\|EV^*U-E\|\le\|E\|\|1-V^*U\|=\|E\|\|V-U\|\le\|E\|(\|V-1\|+\|1-U\|).
$$
It follows that if $\|V-1\|<\frac{r_E}{2\|E\|}$ and $\|U-1\|<\frac{r_E}{2\|E\|}$, then $EV^*U\big|_\s$ is an isomorphism between $\s$ and $\t$. Thus, $\s'$ and $\t'$ have a common complement.
 \end{proof}

As we have observed above the connected components of $\fde_{ij}$ of $\fde$, where $i<\infty$ or $j <\infty$, coincide with the whole corresponding component of $\p(\h) \times \p(\h)$. For the connected component $\fde_\infty$ we have the following.

 \begin{teo}
$\fde_\infty$ is dense in $\p_{\infty, \infty} \times \p_{\infty, \infty}$.  
 \end{teo}
 \begin{proof}
Pick $(P,Q) \in \p_{\infty, \infty} \setminus \fde_\infty$.  Notice that we can take a unitary conjugation of both projections to prove the statement, so by Halmos' model we may assume   
\begin{align*}
P &= 1 \oplus 0 \oplus 1  \oplus 0 \oplus P_0 ,   \\
Q & =1 \oplus 0 \oplus 0  \oplus 1 \oplus Q_0 ,  
\end{align*}
where we follow the order of the subspaces in \eqref{paul h}. The projections $P_0$, $Q_0$ acting on $ \l \times \l$,  are given by
$$
P_0 = \left(\begin{array}{cc} 1 & 0 \\ 0 & 0 \end{array}\right) \ \hbox{ and } \ Q_0 = \left(\begin{array}{cc} C^2 & CS \\ CS & S^2 \end{array}\right),
$$
where $C=\cos(X)$, $S=\sin(X)$, and $\| X \| \leq \frac{\pi}{2}$. 
Then, $(P,Q) \notin \fde$ implies that $X$ is compact (see Remark \ref{Halmos Giol}). 
We consider two cases.

In the first case, we assume that $\l$ is infinite dimensional. 
Recall that the operators $S=\sin(X)$ and $C=\cos(X)$ are injective in Halmos'model, hence we get that neither $0$ nor $\frac{\pi}{2}$ are eigenvalues of $X$. Therefore, its spectral decomposition must be
$$
X=\sum_{k=1}^\infty \lambda_k P_k,
$$
where $\{ \lambda_k\}_{k \geq 1}$ is a sequence of positive eigenvalues in $(0, \frac{\pi}{2})$ that converges to $0$ and $\{ P_k\}_{k \geq 1}$ is the corresponding sequence of finite-rank spectral projections. For $\epsilon >0$ sufficiently small, the operators of the form
$$
X_\epsilon:=\epsilon 1_\l   +\sum_{k \, : \, k \geq 2 \epsilon} \lambda_ k P_k 
$$
have spectrum contained in $(0, \frac{\pi}{2})$. This implies that $S_\epsilon=\sin(X_\epsilon)$ and $C_\epsilon=\cos(X_\epsilon)$ are positive injective operators, and the projections acting on $\l$ defined by
$$
Q_{\epsilon, 0} = \left(\begin{array}{cc} C_\epsilon^2 & C_\epsilon S_\epsilon \\ C_\epsilon S_\epsilon & S_\epsilon ^2 \end{array}\right)
$$
clearly satisfy $Q_{\epsilon, 0} \to Q_0$ in the operator norm as $\epsilon \to 0$. Next take the projections acting on $\h$ defined by
$$
Q_\epsilon:=1 \oplus 0 \oplus 0  \oplus 1 \oplus Q_{\epsilon,0} .
$$
By using that $S_\epsilon$ and $C_\epsilon$ are injective operators, one can easily verify that $R(P_0) \cap R(Q_{\epsilon, 0})=R(P_0) \cap N(Q_{\epsilon, 0})=\{ 0 \}$. This implies $R(P) \cap R(Q_\epsilon)=R(P) \cap R(Q)$ and $R(P) \cap N(Q_\epsilon)=R(P) \cap N(Q)$, and consequently, we obtain the following decomposition
$$R(P)=R(P) \cap R(Q_\epsilon) \oplus R(P) \cap N(Q_\epsilon) \oplus \l ,$$ 
for $\epsilon \geq 0$ small enough. Consider the operators $G_\epsilon=Q_\epsilon|_{R(P)}$ introduced in Remark \ref{choreo}. Since $N(G_\epsilon)^\perp=R(P)\cap R(Q) \oplus \l$ by our previous computations, it follows that 
$(1-G_\epsilon^* G_\epsilon)|_{N(G_\epsilon)^\perp}$ is not compact because $S_\epsilon$, or equivalently $X_\epsilon$,  are not compact for $\epsilon >0$ sufficiently small. We thus get  $(P,Q_\epsilon) \in \fde$ and $(P,Q_\epsilon) \to (P,Q)$.

In the second case, we assume that $\l$ is finite dimensional. Recall that $\l$ is defined by $R(P)=R(P) \cap R(Q) \oplus R(P) \cap N(Q) \oplus \l$. Since $P, Q \in \p_{\infty, \infty}$, it must be $\dim R(P) \cap R(Q)=\infty$ or 
$\dim R(P) \cap N(Q)=\infty$. But $(P,Q) \in \fga$,  so that $\dim R(P) \cap N(Q) \neq \dim N(P) \cap R(Q)$. Then, using that $\h$ is separable, we may assume $\dim R(P) \cap N(Q) < \infty$ because we can interchange the roles of $P$ and $Q$ if necessary. Hence we obtain 
$\dim R(P) \cap R(Q)=\infty.$

 As we have observed in  Remark \ref{comp fin} the sets $\fga$ and $\fde$ are invariant by taking orthogonal complements . Therefore, $(P^\perp, Q^\perp) \in \fga$. There is a corresponding subspace $\l'$ defined by 
 $N(P)=N(P) \cap N(Q) \oplus N(P) \cap R(Q) \oplus \l'$ for the Halmos' model of $(P^\perp,Q^\perp)$. 
 If $\dim \l'=\infty$, then $(P^\perp, Q^\perp)$ can be approximated by pairs in $\fde$ by the first case we have considered, and consequently, the same holds for $(P,Q)$ by taking orthogonal complements. Thus, we assume $\dim \l' < \infty $. By  the same argument given in the previous paragraph, we  also get $\dim N(P) \cap N(Q)=\infty$.  
 
 We take the decomposition $\h=\k_1 \oplus \k_2 \oplus \k_3$, where $\k_1=R(P) \cap R(Q)$, $\k_2=N(P) \cap N(Q)$ and $\k_3$ is the orthogonal complement of the first two subspaces. Since $\k_1$ and $\k_2$ have both infinite dimension and co-dimension, we can  conjugate by a unitary operator the projections $P$ and $Q$, and write them in terms of  $\h=\k_1 \times \k_1 \times \k_3$ as follows
 $$
 P=\begin{pmatrix}        1   &    0   &  0   \\  0   &   0   &   0   \\  0   &    0  &  P_3   \end{pmatrix}, \, \, \, \, \, 
Q=\begin{pmatrix}        1   &    0   &  0   \\  0   &   0   &   0   \\  0   &    0  &  Q_3   \end{pmatrix},
 $$
 for some projections $P_3$ and $Q_3$. Then, we define the following projections  
 $$
 Q_\epsilon := \begin{pmatrix}        \cos^2(\epsilon)    &    \cos(\epsilon) \sin(\epsilon)   &  0   \\   \cos(\epsilon) \sin(\epsilon)   &   \sin^2(\epsilon)    &   0   \\  0   &    0  &  Q_3
 \end{pmatrix}, \, \, \, \, \epsilon > 0.
 $$
Clearly, we have $Q_\epsilon \to Q$. Take the corresponding     Gramian operators  $G_\epsilon$ for $P$ and $Q_\epsilon$. Since   $ N(G_\epsilon)^\perp=\k_1 \oplus \m$, for some subspace $\m$ of $\k_3$, and $1-\cos^2(\epsilon)=\sin^2(\epsilon)=\sin^2(\epsilon) 1_{\k_1}$ is not compact because $\dim \k_1=\infty$, we find that $(1- PQ_\epsilon P)|_{N(G_\epsilon)^\perp}$ is neither compact. We conclude that  $(P,Q_\epsilon) \in \fde$ for all sufficiently small $\epsilon >0$.
 \end{proof}
 



\section{Pairs of subspaces without a common complement}\label{without com compl}. 

\vspace{-1cm}

\subsection{Geometric structure}

Let us consider now 
$$
{\bf\Gamma}:=\{(P_\s,P_\t)\in\p(\h)\times \p(\h): \s, \t \hbox{ do not have a common complement}\}.
$$

\begin{rem}\label{fin or inf gamma}
According to Remark \ref{comp fin}, we have that
$$
\fga_{ijkl}:=\fga \cap (\p_{i,j} \times \p_{k,l})=\p_{i,j} \times \p_{k,l}, 
$$
whenever $i \neq k$  or $j \neq l$, are the only connected components of $\fga$ with finite dimensional rank or corank. Hence we are left to understand the structure of pairs in $\p_{\infty, \infty} \times \p_{\infty, \infty}$ without a common complement. 
\end{rem}

The facts recalled in Remark \ref{choreo} $3.$ can be used to rewrite 
$$
{\bf\Gamma}=\{(P_\s,P_\t): P_\s P_\t^\perp \hbox{ is compact in } (\s\cap\t^\perp)^\perp \hbox{ and } \dim\s\cap\t^\perp\ne\dim\s^\perp\cap\t\}.
$$
Denote by  $P_0$ and $Q_0$ the reductions of $P_\s$ and $P_\t$ to $\h_0$. We shall see below (in Subsection \ref{other}) that  the condition 
\begin{equation*}\label{condicion 1}
P_\s P_\t^\perp \hbox{ is  compact in } (\s\cap\t^\perp)^\perp,
\end{equation*}
can be replaced by the condition  
\begin{equation*}\label{condicion 2}
A_0:=P_0-Q_0 \hbox{   is compact;}
\end{equation*} 
or by the condition
\begin{equation*}\label{condicion 3}
P_0Q_0^\perp \hbox{ is compact;}
\end{equation*} 
or also by
\begin{equation*}\label{condicion 4}
\hbox{ either } P_\s P_\t^\perp \hbox{  or } P_\t P_\s^\perp \hbox{  is compact.}
\end{equation*}  

Given a $\l$ is a Hilbert space, we denote by $\k(\l)\subset\b(\l)$ the ideal of compact operators in $\l$, and by
$$
\pi_\l:\b(\l)\to \b(\l)/\k(\l):={\bf C}(\l)
$$
the $*$-epimorphism onto the Calkin algebra ${\bf C}(\l)$. For a projection $P$, denote by $r(P)=\dim R(P)\le +\infty$ the rank of $P$.  
In \cite{pq compacto}, the set 
$$
\c=\{(P,Q)\in\p(\h)\times\p(\h): PQ\in\k(\h)\}
$$ 
was studied. 
\begin{rem}
We recall  the basic facts from \cite{pq compacto} needed here:
\noindent

\begin{enumerate}
\item
Let $(P,Q)\in\c$ such that both $P$ and $Q$ have infinite rank. The projections $\pi(P)$ and $\pi(Q)$ are mutually orthogonal, and therefore can be written as $2\times 2$ matrices in terms of $\pi(P)$ as
$$
\pi(P)=\left( \begin{array}{cc} 1 & 0 \\ 0 & 0\end{array} \right) \hbox{ and }\  \pi(Q)=\left( \begin{array}{cc} 0 & 0 \\ 0 & q\end{array} \right) .
$$
Since  both $P$ and $Q$ have infinite rank,  we have (in terms of the above $2\times 2$ matrix description): 
$$
\c_1:=\{(P,Q): q=1\hbox{ in } {\bf C}(R(P)^\perp)\},
$$
and
$$
\c_\infty:=\{(P,Q): q \hbox{ is a proper projection } (\ne 0,1) \hbox{ in } {\bf C}(R(P)^\perp)\}.
$$
The elements $(P,Q)$ in the class $\c_1$ have finite Fredholm index.
\end{enumerate}
\end{rem}

We recall the following results from \cite{pq compacto}.
\begin{itemize}
\item (Theorems 6.5 and 7.6)
The connected components of $\c$ are 
$$
\c_\infty \ \hbox{ and } \  \c_1^n=\{(P,Q)\in\c_1: \hbox{index}(P,Q)=n\}.
$$ 
\item (Theorem 8.5)
the set $\c$ is a $C^\infty$ (non complemented) submanifold of $\b(\h)\times\b(\h)$. Therefore $\c_\infty$ $\c_1$ are sumanifolds of $\b(\h)\times\b(\h)$.
\item (Proposition 6.6)
Let $(P,Q)\in\c$, then $(P,Q)\in\c_1$ if and only if $\dim N(P)\cap N(Q)<\infty$.
\end{itemize}

Then we have the following:
\begin{teo}\label{teo 62}
Let $\s, \t$ be closed subspaces of both infinite dimension and codimension.
The following are equivalent
\begin{enumerate}
\item
$(P_\s,P_\t)\in{\bf\Gamma}$ with  
$$
\dim \s\cap\t^\perp<\infty \ \hbox{ and } \ \dim \s^\perp\cap\t<\infty.
$$
\item 
$(P_\s,P_\t^\perp)\in\c_1^n$ with $n\ne 0$.
\end{enumerate}
\end{teo}
\begin{proof}
2) $\implies $ 1): By Proposition 6.6 in \cite{pq compacto} (cited above), we have that $(P_\s,P_\t^\perp)\in\c_1$ if and only if
$$
\dim N(P_\s)\cap N(P_\t^\perp)=\dim \s^\perp\cap\t<\infty.
$$
On the other hand, the fact that $P_\s P_\t^\perp$ is compact forces that the intersections of the ranges (where this product acts as the identity) be of finite dimension:
$$
\dim R(P_\s)\cap R(P_\t^\perp)=\dim \s\cap\t^\perp<\infty.
$$
Since the index $n \ne 0$, these numbers are different, and therefore $\s$ and $\t$ do not have a common complement.
1) $\implies$ 2): straightforward, using the above computations.
\end{proof}
Denote by 
\begin{equation*}\label{Gamma1}
{\bf\Gamma}_1=\{(P_\s,P_\t)\in{\bf\Gamma}: (P_\s,P_\t^\perp)\in\c_1\},
\end{equation*}
which due to the above result above coincides with
$$
{\bf \Gamma}_1=\{P_\s,P_\t)\in\Gamma: \dim \s\cap\t^\perp<\infty \hbox{ and } \dim \s^\perp\cap\t<\infty\}.
$$
Moreover, put
\begin{equation*}\label{Gamma1n}
{\bf\Gamma}_1^n=\{(P_\s,P_\t)\in{\bf\Gamma}: (P_\s,P_\t^\perp)\in\c_1^n\}.
\end{equation*}
\begin{coro}
${\bf\Gamma}_1$ is a $C^\infty$ (non complemented) submanifold of $\p(\h)\times\p(\h)$. Its connected components are ${\bf\Gamma}_1^n$, $n\in\mathbb{Z}$.
\end{coro}
\begin{proof}
Straightforward, using the previous theorem, and the results from \cite{pq compacto} cited above: the global diffeomorphism $\mathfrak{t}:\p(\h)\times\p(\h)\to\p(\h)\times\p(\h)$,
$$
\mathfrak{t}(P,Q)=(P,Q^\perp)
$$
 maps ${\bf\Gamma}_1^n$ onto $\c_1^n$.
\end{proof}
We have the following elementary corollaries to Theorem \ref{teo 62}:
\begin{coro}
$(P,Q)\in{\bf\Gamma}_1^n$ if and only if $(Q,P)\in{\bf\Gamma}_1^{-n}$
\end{coro}
\begin{proof}
We have seen that
$$
{\bf\Gamma}_1=\{(P_\s,P_\t)\in{\bf\Gamma}: \dim \s\cap\t^\perp<\infty, \   \dim \s^\perp\cap\t<\infty\}.
$$
Clearly these conditions are symmetric, and also 
$$
\hbox{index}(P_\s,P_\t)=-\hbox{index}(P_\t,P_\s).
$$
\end{proof}

Recall from Section \ref{section 2} the definition of Fredholm pairs.
\begin{coro}
$(P_\s,P_\t)\in {\bf\Gamma}_1^n$ if and only if $(P,Q)$ is a Fredholm pair with index $n$.
\end{coro}
\begin{proof}
Clearly we have that $(P_\s,P_\t)\in{\bf\Gamma}_1$ if and only if $(P_\s,P_\t^\perp)\in\c$ and $(P_\t,P_\s^\perp)\in\c$ are compact, i.e., 
$$
P_\s-P_\s P_\t P_\s \ \hbox{ and } \ P_\t-P_\t P_\s P_\t \hbox{ are compact},
$$
i.e., 
$$
P_\t\big|_\s:\s\to\t 
$$
is a Fredhom operator. Its index is $\dim\s\cap\t^\perp-\dim\s^\perp\cap\t$.
\end{proof}

Let us now consider the complementary part
$$
{\bf\Gamma}_\infty:=\{(P_\s,P_\t)\in\Gamma: \dim \s\cap\t^\perp =+\infty\  \hbox{ or }\  \dim \s^\perp \cap\t=+\infty\}.
$$
Clearly only one of the two dimensions can be infinite in each case, and then this set parts into two disjoint subsets
$$
{\bf\Gamma}_\infty={\bf\Gamma}_\infty^l\cup{\bf\Gamma}_\infty^r,
$$
where
$$
{\bf\Gamma}_\infty^l:=\{(P_\s,P_\t)\in \p(\h)\times \p(\h): \dim \s\cap\t^\perp<\infty, \dim\s^\perp\cap\t=+\infty\}
$$
or equivalently
$$
{\bf\Gamma}_\infty^l=\{(P_\s,P_\t)\in{\bf\Gamma}: P_\s P_\t^\perp \hbox{ compact and } P_\t P_\s^\perp \hbox{ non compact}\},
$$
and 
\begin{align*}
{\bf\Gamma}_\infty^r & :=\{(P_\s,P_\t)\in  \p(\h)\times \p(\h): \dim \s^\perp\cap\t<\infty, \dim\s\cap\t^\perp=+\infty\}\\
&=\{(P_\s,P_\t)\in{\bf\Gamma}: P_\t P_\s^\perp \hbox{ compact and } P_\s P_\t^\perp \hbox{ non compact}\}.
\end{align*}
\begin{coro}
Both sets ${\bf\Gamma}_\infty^\l$, ${\bf\Gamma}_\infty^r$ are $C^\infty$ (non complemented) submanifolds of $\p(\h)\times\p(\h)$, which are diffeomorphic to $\c_\infty$. Namely:
$$
(P_\s,P_\t)\mapsto(P_\s,P_\t^\perp) \hbox{ maps } {\bf\Gamma}_\infty^l \hbox{ onto } \c_\infty
$$
and
$$ 
(P_\s,P_\t)\mapsto(P_\s^\perp,P_\t) \hbox{ maps } {\bf\Gamma}_\infty^l \hbox{ onto } \c_\infty.
$$
\end{coro}
\begin{proof}
Use Theorem 8.5 of \cite{pq compacto}.
\end{proof}
\begin{rem}
The pairs $(P_\s,P_\t)\in{\bf\Gamma}_\infty^l$ are also characterized by the condition that
$P_\s\big|_\t\to\s$ is a semi-Fredholm operator of index $+\infty$. Indeed, 
$P_\s-P_\s P_\t P_\s$ is compact, and thus $P_\s P_\t P_\s$ is a compact perturbation of the identity, as an operator in $\s$; i.e., $P_\t|_\s:\s\to\t$ has a left inverse, but not a right inverse.

Similarly, $(P_\s,P_\t)\in{\bf\Gamma}_\infty^r$ if and only if $P_\t\big|_\s:\s\to\t$ has a right inverse, but not a left inverse. 
\end{rem}

Consequently, we have the following result. 
\begin{coro}
The connected components of ${\bf\Gamma}_\infty$ are ${\bf\Gamma}_\infty^l$ and ${\bf\Gamma}_\infty^r$.
\end{coro}
\begin{proof}
Let $(P_{\s(t)},P_{\t(t)})$ be a continuous path in ${\bf\Gamma}_\infty$, $t\in[0,1]$. In particular, $P_{\s(t)}$ and $P_{\s(t)}$ are continuous paths in $\p(\h)$. Since for each $P\in\p(\h)$  the map 
$$
\u(\h)\to\p(\h), \ U\mapsto UPU^*
$$
is a fiber bundle, there exist continuous curves $U(t), V(t)$ in $\u(\h)$, $t\in[0,1]$, such that 
$$
U(t)P_{\s(0)}U^*(t)=P_{\s(t)} \ \hbox{ and } \ V(t)P_{\t(0)}V^*(t)=P_{\t(t)} \ \hbox{ for } t\in[0,1].
$$
In particular, $U(t)$ maps $\s(0)$ onto $\s(t)$, and $V(t)$ maps $\t(0)$ onto $\t(t)$. Then 
$$
U^*(t)\left(P_{\s(t)}\big|_{\t(t)}\right)V(t)=U^*(t)P(t)V(t)\big|_{t(0)}:\t(0)\to\s(0)
$$
is a continuous path of semi-Fredholm operators between the (fixed) spaces $\t(0)$ and $\s(0)$. It folllows that the index (be it $+\infty$ or $-\infty$) remains constant, i.e. the curve $(P_{\s(t)},P_{\t(t)})$ remains entirely inside ${\bf\Gamma}_\infty^l$ or ${\bf\Gamma}_\infty^r$. 

Clearly, since both ${\bf\Gamma}_\infty^l$ and ${\bf\Gamma}_\infty^r$ are homeomorphic to $\c_\infty$, they are connected.
\end{proof}

\begin{prop}
The subsets ${\bf\Gamma}_1^n$ {\rm(}$n\ne 0${\rm )}, ${\bf\Gamma}_\infty^l$ and ${\bf\Gamma}_\infty^r$ are relatively open in ${\bf\Gamma}$.
\end{prop}
\begin{proof}
Fix $(P_{\s_0},P_{\t_0})\in{\bf\Gamma}_1^n$, for $n\ne 0$. Let $(P_\s,P_\t)\in{\bf\Gamma}_1$ such that $\|P_\s-P_{\s_0}\|<r<1$ and $\|P_\t-P_{\t_0}\|<r<1$. Then there exist unitaries $U_0, V_0$ with $\|U_0-1\|, \|V_0-1\|<\delta_r$, and $\delta_r \to 0$ if $r\to 0$, such that $U_0\s_0=\s$ and $V_0\t_0=\t$ (i.e., $U_0P_{\s_0}U_0^*=P_\s$ and $V_0P_{\t_0}V_0^*=P_\t$). Then, since $P_\s$ determines a Fredholm operator between $\t$ and $\s$, $U_0^*P_\s V_0$ determines a Fredholm operator between $\t_0$ and $\s_0$, with the same index. Note that (since $U_0^*P_\s U_0=P_{\s_0}$)
$$
\|U_0^*P_\s V_0-P_{\s_0}\|\le\|U_0^*P_\s V_0-U_0^*P_\s U_0\|=\|(U_0^*P_\s)(U_0-V_0)\|\le\|U_0-V_0\|
$$
$$
\le \|U_0-1\|+\|1-V_0\|<2\delta_r.
$$
In particular, this implies that  $\|U_0^*P_\s V_0\big|_{\t_0}-P_{\s_0}\big|_{\t_0}\|<2\delta_r$ in $\b(\s_0,\t_0)$. This implies that for $r$ small enough (depending on $P_{\s_0}\big|_{\t_0}$),  $U_0^*P_\s V_0\big|_{\t_0}$ has the same index as  $P_{\s_0}\big|_{\t_0}$. It follows that (for such values of $r$) $(P_\s,P_\t)\in {\bf\Gamma}_1^n$. The argument for the components ${\bf\Gamma}_\infty^l$ and ${\bf\Gamma}_\infty^r$ is similar.
\end{proof}
\begin{coro}
${\bf\Gamma}$ is a smooth (non complemented) submanifold of $\p(\h)\times\p(\h)$.
\end{coro}

\subsection{Other characterizations of $\fga$}\label{other}

We present other characterizations of $\fga$ in terms of sums or products of projections. 
In \cite{davis} Chandler Davis studied operators $A=P-Q$, where $P$ and $Q$ are orthogonal projections. They are selfadjoint contractions with certain spectral symmetries which we describe below. 
\begin{rem}
Let $P,Q\in\p(\h)$ and denote $A=P-Q$. The first two facts are elementary. The third asertion can be found in \cite{davis}.
\begin{itemize}
\item
$N(A)=N(P)\cap N(Q)\oplus R(P)\cap R(Q)$.
\item
$N(A-1)=R(P)\cap N(Q)$ and $N(A+1)=N(P)\cap R(Q)$.
\item
Denote by $\h_0=\h\ominus \left( N(A)\oplus N(A-1)\oplus N(A+1)\right)$ (usually named the {\it generic part} of $P$ and $Q$). Note that $\h_0$ reduces ($P$, $Q$ and) $A$, denote by $A_0=A\big|_{\h_0}$. Then Davis proved that there exists a symmetry $V=V_{P,Q}$ (by  symmetry we mean an operator $V=V^*=V^{-1}$, $VPV=Q$ and $VQV=P$). In particular, $VA_0V=-A_0$. The spectrum $\sigma(A_0)$ of $A_0$ is contained in $[-1,1]$. The existence of this symmetry $V$ implies that $\sigma(A_0)$ is symmetric with respect to the origin, $\lambda\in\sigma(A_0)$ if and only if $-\lambda\in\sigma(A_0)$, and the multiplicity function is symmetric. 
\end{itemize}
The spectral  symmetry of the reduction $A_0$ may break for $A$ at the extremes of the spectrum: $N(A-1)$ and $N(A+1)$ may have different dimensions.

Chandler Davis in fact showed much more: that the decompositions of $A_0$ (as diference of projections) are in one to one correspondence with the operators $V$ implementing this symmetry, with explicit formulas to recover $P$ and $Q$ from $V$.
\end{rem}

We claim that the condition that $\s$ and $\t$ admit a common complement is in fact a property of the operator $A=P_\s-P_\t$. To this purpose, we shall need the following facts from \cite{pq vs p-q}, which relate the difference  $A$ with the product $B:=P_\s P_\t$.

We shall say that $ T$ is S-decomposable if it has a  singular value (or Schmidt) decomposition,
\begin{equation}\label{S}
T=\sum_{n\ge 1} s_n\langle\ \ \ , \xi_n\rangle\psi_n=\sum_{n\ge 1}s_n \psi_n\otimes\xi_n,
\end{equation}
where $\{\xi_n: n\ge 1\}$ and $\{\psi_n: n\ge 1\}$ are orthonormal systems, and $s_n>0$. In this case, $\{\psi_n\}$, $\{\xi_n\}$ are orthonormal bases of $\overline{R(T)}$, $N(T)^\perp$, respectively and $T\xi_n=s_n\psi_n$, $T^*\psi_n=s_n\xi_n$, $T^*T\xi_n=s_n^2\xi_n$, $TT^*\psi_n=s_n^2\psi_n$ for all $n\ge 1$.

\begin{rem} (see \cite[Thms. 2.2, 2.4 and Rem. 2.3]{pq vs p-q})
Let $P,Q$ be orthogonal projections. 
\begin{enumerate}
\item
Suppose that $B=PQ$ is S-decomposable with singular values $s_n$. Then $A=P-Q$ is diagonalizable, with eigenvalues $\pm(1-s_n^2)^{1/2}$, $n\ge 1$, and maybe $0, -1$ and $1$.
\item
If $A=P-Q$ is diagonalizable with (non zero) eigenvalues $\pm\lambda_n$ ($0<|\lambda_n|<1$) and $\pm 1$ , then $B=PQ$ is S-decomposable with singular values $(1-\lambda_n^2)^{1/2}$ and $1$.
\item
Except for $1$ and $-1$, the eigenvalues $(1-s_k^2)^{1/2}$ and $-(1-s_k^2)^{1/2}$ of $A$ have the same multiplicity. For the eigenvalues in $(-1,1)$ (or the singular values in $(0,1)$ we have the following relation: the multiplicity of $(1-\lambda^2_n)^{1/2}$ as a singular value of  $B$  is $2$ times the multiplicity of $\lambda_n$ as an eigenvalue of $A$. 
\end{enumerate}
\end{rem}
\begin{rem} 
An elementary observation concerning this property, is that $PQ$ is S-decomposable if and only if $P(1-Q)$ also is. And since clearly in this case also $QP$ has the same property, it follows that also $(1-P)Q$ and $(1-P)(1-Q)$ are S-decomposable (see \cite[Prop. 2.7]{pq vs p-q}). Therefore the following condition can be added to the list above:

\noindent

4. (see \cite[Corol. 2.8]{pq vs p-q}) $P-Q$ is diagonalizable if and only if $P+Q$ is diagonalizable.
In that case, $\lambda_n$ is an eigenvalue of $P-Q$ with $0<|\lambda_n|<1$  if and only if $1\pm (1-\lambda_n)^2$ is an eigenvalue of $P+Q$, with the same multiplicity. 
\end{rem}
Putting these fact together, we have the following:
\begin{teo}\label{A y B}
Let $\s,\t$ be closed subspaces of $\h$. Then $\s$ and $\t$  do not have a common complement if and only if  the following conditions hold:
\begin{enumerate}
\item
$\dim N(A-1)\ne \dim N(A+1)$
\item
$A_0=A\big|_{\h_0}$ is compact, where $\h_0=\left(N(A-1)\oplus N(A+1)\right)^\perp$.
\end{enumerate}
\end{teo}
\begin{proof}
The proof is just the translation of the two conditions proven by Lauzon and Treil \cite{lauzontreil}, transcribed here in Remark \ref{choreo}.3.
The first condition is just the translation of $\dim\s\cap\t^\perp\ne \dim\s^\perp\cap\t$.
With respect to the second condition (compacteness of $1-G^*G$ as an operator on $N(G)^\perp$) we use Halmos' decomposition:
$$
\s\cap\t \oplus \s^\perp\cap\t^\perp\oplus \s\cap\t^\perp \oplus \s^\perp\cap\t\oplus\h_0.
$$
Note that $1-G^*G$ in $\s$ coincides with $P_\s-P_\s P_\t P_\s=P_\s P_\t^\perp P_\s$ in $\s$. Then clearly $1-G^*G$ is compact in  $\s\ominus (\s\cap\t^\perp)$ if and only  $P_\s P_\t^\perp P_\s$ is compact in $\s^\perp \oplus (\s\ominus (\s\cap\t^\perp))=(\s\cap\t^\perp)^\perp$, (because $P_\s P_\t^\perp P_\s$ is trivial in $\s^\perp$).
Note that $P_\s P_\t^\perp P_\s$ in the decomposition
$$
(\s\cap\t^\perp)^\perp=\s\cap\t \oplus \s^\perp\cap\t^\perp\oplus \s^\perp\cap\t \oplus\h_0,
$$
is 
$$
P_\s P_\t^\perp P_\s=0 \oplus 0 \oplus 0 \oplus P_0 Q_0^\perp P_0,
$$
where $P_0$, $Q_0$ denote the reductions of $P_\s$ and $P_\t$ to $\h_0$, respectively.
Therefore the second condition is equivalent to $P_0 Q_0^\perp P_0$ being compact. This is equivalent to $P_0 Q_0^\perp$ being compact.
This in turn is equivalent to $P_0 Q_0^\perp$ being S-decomposable with singular values $s_k$ such that  $s_k$ has finite multiplicity and the sequence is finite or else $s_k\to 0$. By the above facts, this is equivalent to 
$$
P_0+Q_0^\perp=A_0+1
$$
being diagonalizable, with eigenvalues $\lambda_k$ of finite multiplicity, and the relations above (between eigenvalues of the sum  of two projections, namely $P_0$ and $Q_0^\perp$, the eigenvalues of the difference and the singular values of the product) mean that $\lambda_k$ have finite multiplicity,  are finite or else $\lambda_k\to 1$ (an important observation here is that neither $+1$ nor $-1$ are eigenvalues of $A_0$, because $P_0$ and $Q_0$ are in generic position). This is clearly equivalent to $A_0$ being diagonalizable with eigenvalues $\alpha_k$ of finite multiplicity, being finitely many, or else $\alpha_k\to 0$, i.e., $A_0$ compact.
\end{proof}
\begin{rem}\label{producto compacto}
One can further elaborate on the argument of the above theorem. Suppose that $\s$ and $\t$ do not have a common complement. Then since $\s\cap\t^\perp$ and $\s^\perp\cap\t$ have different dimensions, one of the two dimensions must be finite. If $\dim \s\cap\t^\perp<\infty$, then it follows that (with the argument above) $P_\s P_\t^\perp P_\s$ (or equivalently, $P_\s P_\t^\perp$)  is compact in in the whole space $\h$. Similarly, if $\dim \s^\perp\cap\t <\infty$, $P_\t P_\s^\perp$ is compact. 

Summarizing, we have
\begin{coro}
$\s$ and $\t$ do not have a common complement if and only if \begin{itemize}
\item
$\dim \s\cap\t^\perp\ne \dim \s^\perp\cap\t$, and
\item
either $P_\s P_\t^\perp$ or $P_\t P_\s^\perp$ is compact.  
\end{itemize}
\end{coro}
\end{rem}

 Let us cite the following result from \cite{pq compacto} which is relevant to our situation:

\begin{rem}(\cite[Thm. 4.1]{pq compacto}) Let $P,Q$ be orthogonal projections. Then $PQ$ is compact if and only if there exist orthonormal basis $\{\xi_n: n\ge 1\}$ of $R(P)$ and  $\{\psi_n: n\ge 1\}$ of $R(Q)$  such that 
$$
\langle\xi_n,\psi_k\rangle=0 \ \hbox{  if  } \ n\ne k
$$
and 
$$
\langle\xi_n, \psi_n\rangle\to 0 \ \  {\rm(}n\to \infty{\rm)}.
$$
In the literature (see for instance \cite{folland}), in the context of example \ref{incertidumbre} below, these are called {\it bi-orthogonal bases}.
\end{rem}

\begin{coro}
Let $\s,\j$ be closed subspaces of $\h$. Then there exist no common complements for $\s$ and $\t$ if and only if 
$\dim \s\cap\t^\perp\ne \dim \s^\perp\cap\t$ and, if the first of these two numbers is finite there exist orthonormal bases $\{\xi_n: n\ge 1\}$ of $\s$ and  $\{\psi_n: n\ge 1\}$ of $\t^\perp$ such that 
$$
\langle\xi_n,\psi_k\rangle=0 \ \hbox{  if  } \ n\ne k
$$
and 
$$
\langle\xi_n, \psi_n\rangle\to 0 \ \  {\rm(}n\to \infty{\rm)}.
$$
If instead $\dim \s^\perp\cap\t<\infty$, then there exist ohonormal bases of $\s^\perp$ and $\t$ with the analogous property. 
\end{coro}
\begin{proof} The condition $P_\s P_\t^\perp P_\s$ compact is equivalent to $P_\s P_\t^\perp$ compact, and thus it the result in the above remark applies verbatim.
\end{proof}

\section{Geodesics in $\Delta$ and ${\bf\Gamma}$}\label{geodesicas en delta y gamma}

  We first observe the following facts about geodesics and common complements.

\begin{rem}
\noi $1.$ The condition that $\dim N(G)=\dim \s\cap\t^\perp$ equals $\dim N(G^*)=\dim \s^\perp\cap\t$ (which is stronger than condition (\ref{la condicion})) is equivalent to the existence of a geodesic of $\p(\h)$ joining $P_\s$ and $P_\t$. Moreover,   following the explicit construction of these geodesics (see for instance \cite{p-q}), it can be shown that if $P(t)$ is a minimal geodesic joining $P(0)=P_\s$ with $P(1)=P_\t$, then for any pair $t, t'\in[0,1]$, $R(P(t))$ and $R(P(t'))$ have a common complement.

\medskip

\noi $2.$ As the condition for the existence of a geodesic is stronger than the condition for a common complement, it follows that two projections $P_\s$, $P_\t$ such that  $(P_\s$, $P_\t) \in \fde$ may or may not be joined by a geodesic. However, in this case one can always join $P_\s$ and $P_\t$ by a polygonal consisting in two minimizing geodesics in $\p(\h)$. This is a consequence of Giol's characterization in Remark \ref{Charact Giol} $i)$. Indeed, the existence of a projection $P$ such that $\| P_\s - P\|<1$ and $\| P- P_\t\|<1$ implies that there is a minimizing geodesic joining $P_\s$ and $P$, and another minimizing geodesic joining $P$ and $P_\t$.

\end{rem}

Since the sets $\fde$ and $\fga$ are included in $\p(\h) \times \p(\h)$, we may also ask for the existence of minimizing geodesics joining points of $\fde$ or $\fga$. This leads us to study  when geodesics of $\p(\h)\times \p(\h)$ remain  in 
$\fde$ or $\fga$ along their paths.   We endow $\p(\h) \times \p(\h)$ with the product Finsler metric. The length of a piecewise $C^1$ curve 
$\delta:[0,1] \to \p(\h) \times \p(\h)$, $\delta=(\delta_1 , \delta_2)$, is given by
$$
L(\delta)=\int_0^1 (\| \dot{\delta_1} \|^2  + \| \dot{\delta_2}   \|^2 )^{1/2} dt.
$$
Thus, the curve  $\delta$ is a (minimizing) geodesic in $\p(\h) \times \p(\h)$ if and only if $\delta_1$ and $\delta_2$ are (minimizing) geodesics in $\p(\h)$.

\begin{rem}
We first discuss the case of projections with range or nullspace of finite dimension. 
\begin{itemize}
\item[1.] By Remark \ref{comp fin} the connected components with finite rank or corank of $\fde$ are given by $\fde_{ij}=\p_{i,j} \times \p_{i,j}$  ($i <\infty$ or $j <\infty$). In this case every pair of points in $\p_{i,j} \times \p_{k,l}$ can be joined by  a minimizing geodesic contained in $\fde$.
\item[2.] By Remark \ref{fin or inf gamma} for the case of $\fga$ we known that these connected components are given by
$
\fga_{ijkl}=\p_{i,j} \times \p_{k,l}$, when $i \neq k$  or $j \neq l$. 
If $i \neq k$, then there is a minimizing geodesic contained in $\fga$ joining two points in $\p_{i,j} \times \p_{k,l}$  if and only there is a minimizing geodesic joining them in  $\p(\h)\times \p(\h)$. Indeed, given   $(P_0,Q_0), (P_1, Q_1) \in \p_{i,j} \times \p_{k,l}$, the  condition for the existence of a geodesic joining them does not necessarily hold. But when this condition is satisfied, any geodesic must be contained in $\fga$, since $\corank (P_t)=\corank (P_0)$ and $\corank (Q_t)=\corank( Q_0)$ for all $t$. The case with $j\neq l$ follows again by using the map $\perp$.
\end{itemize}
\end{rem}

Now we turn to the case of projections in $\p_{\infty,\infty} \times \p_{\infty,\infty}$. In contrast with the previous situation, the following example illustrates that, in general, there is no minimizing geodesic contained in $\fde_\infty$ joining two given points. Furthermore, the example exhibits arbitrary close points in $\fde_\infty$ without a minimizing geodesic (contained in $\fde_\infty$) that  joins them.



 \begin{ejem}\label{un ejemplo}
We begin by giving a pair of projections $(P_0,Q_0)\in \fga$. 
Let $\l$ be an infinite dimensional closed subspace such that $\l \times \l=\h$. Take $c$ an orthogonal projection on $\l$ such that $\dim R(c)=\dim N(c)=\infty$. 
We set
$$
P_0=\begin{pmatrix} 1  & 0  \\ 0  & 0  \end{pmatrix}, \, \, \, \, \, Q_0=\begin{pmatrix} 1  & 0  \\ 0   & c  \end{pmatrix}. 
$$
Since  $P_0-P_0Q_0P_0=0$, $\dim R(P_0) \cap N(Q_0)= 0$ and $\dim R(Q_0) \cap N(P_0)=\dim R(c) =\infty$, we  know that $(P_0,Q_0) \in \fga$. 

Now we construct a geodesic $\delta(t)=(P_t,Q_t)$, $t \in \RR$, such that $\delta(0)=(P_0,Q_0)$. Take an operator $ Z \geq 0$ on $\l$ and put $c^\perp=1-c$, then we define the  following antihermitian operators
$$
X=\begin{pmatrix}  0  & - Z  \\  Z   &  0 \end{pmatrix}, \, \, \, \, \, \, \, \, \, \, Y=\begin{pmatrix}  0   &  -\frac{\pi}{2}c^\perp \\    \frac{\pi}{2}c^\perp   &    0  \end{pmatrix}.
$$
 For $t \in \RR$, we have 
the unitary operators
$$
e^{tX}=\begin{pmatrix} \cos(tZ)   &   -\sin(tZ)  \\   \sin(tZ)   &   \cos(tZ) \end{pmatrix}, \, \, \, \, \, 
e^{tY}=\begin{pmatrix}c+ \cos(t\frac{\pi}{2}) c^\perp   &   -\sin(t\frac{\pi}{2}) c^\perp  \\   \sin(t\frac{\pi}{2})  c^\perp  &  c+  \cos(t\frac{\pi}{2}) c^\perp \end{pmatrix}.
$$ 
According to the results in \cite{p-q}, since $X$ is a $P_0$-codiagonal operator and $Y$ is a $Q_0$-codiagonal operator,  two geodesics of $\p(\h)$ passing through the point  $(P_0,Q_0)$ are given by
\begin{align*}
P_t & :=e^{tX}P_0 e^{-tX}=\begin{pmatrix}   \cos^2(tZ)    &   \cos(tZ)\sin(tZ) \\  \cos(tZ)\sin(tZ) & \sin^2(tZ) \end{pmatrix},\\
Q_t & := e^{tY} Q_0  e^{-tY}=\begin{pmatrix}  c+ \cos^2(\frac{\pi}{2}t) c^\perp   &   \cos(\frac{\pi}{2}t)\sin(\frac{\pi}{2}t)  c^\perp\\  \cos(\frac{\pi}{2}t)\sin(\frac{\pi}{2}t) c^\perp & c+ \sin^2(\frac{\pi}{2}t) c^\perp \end{pmatrix}.
\end{align*}
 In what follows, we further assume  $t \in (-\frac{1}{2},\frac{1}{2})$, $t\neq 0$, $cZ=Zc$ and $\sigma( Z ) \subseteq (0,\frac{\pi}{2})$. We observe the following facts:

 \medskip
 
 \noi 1. $R(P_t)\cap N(Q_t)=\{ 0 \}$. To see this, set $\la_t=\sin(\frac{\pi}{2}t)\cos(\frac{\pi}{2}t)^{-1}$, and note  
 \begin{align*}
 N(Q_t) &=N(Q_0 e^{-tY})=\left\{   \begin{pmatrix} f  \\  g  \end{pmatrix} : \begin{pmatrix}  c+ \cos(t\frac{\pi}{2}) c^\perp   & -  \sin(t\frac{\pi}{2}) c^\perp \\ 0 & c   \end{pmatrix} \begin{pmatrix} f  \\  g  \end{pmatrix}= \begin{pmatrix} 0  \\  0  \end{pmatrix}   \right\}\\
& =\left\{   (\la_t g ,g ) :  g \in N(c)    \right\}.
\end{align*}
Since $\cos(tZ)$ and $\sin(tZ)$ are invertible operators by the previous assumptions on $t$ and $Z$, we set $T_t=\sin(tZ)^{-1}\cos(tZ)$ to get
\begin{align*}
R(P_t)  & =  R(e^{tX}P_0)= \left\{   \begin{pmatrix}  \cos(tZ) &  0 \\ \sin(tZ)   &  0  \end{pmatrix}  
\begin{pmatrix}    f \\ g     \end{pmatrix}  : f,g \in \l \right\} \\ 
& =\{   (T_t g, g) :  g \in \l \}.
\end{align*}
 Using that  $\la_t >0$ and $T_t>0$, we get
 $$
 R(P_t)\cap N(Q_t)=\{  (T_t g,g) : g \in N(c), \, g \in N(T_t - \la_t) \}=\{  0 \}.
 $$
To prove the last equality, note that $T_t g=\la_t g$ if and only if $$\cos(tZ) \sin^{-1}(tZ)g =\sin(\frac{\pi}{2}t)\cos^{-1}(\frac{\pi}{2}t) g.$$ By the spectral theorem, this means that there is $\mu \in \sigma(Z)$ such that 
$1=\tan(t\frac{\pi}{2})\tan(t \mu)$. But this condition does not hold since $t\frac{\pi}{2},  t \mu \in (-\frac{\pi}{4}, \frac{\pi}{4})$.
 
 \medskip
 
 \noi 2. $\dim (N(P_t)\cap R(Q_t))=\dim R(c)=\infty$.  For we note that 
 \begin{align*}
R(Q_t) &=N(Q_t)^\perp=(R(c)\times R(c)) \oplus \{  (f,-\la_t f) : f \in N(c) \},  \\
N(P_t) & = R(P_t)^\perp = \{  (f, -T_t f)  :  f \in \l  \}.
 \end{align*}
 Recalling that $cZ=Zc$, we find that
 \begin{align*}
 N(P_t)\cap R(Q_t)& = \{  (f,-T_tf) :  f \in N(c^\perp (T_t + \la_t)) \}\\
 & \supseteq  \{  (f, -T_t f)  : f \in R(c) \},
 \end{align*}
which gives $\dim (N(P_t)\cap R(Q_t))=\infty$.
 
 \medskip
 
 \noi 3.  Set $G_t=Q_t|_{R(P_t)}: R(P_t) \to R(Q_t)$. Notice that $N(G_t)^\perp=(R(P_t) \cap N(Q_t))^\perp=R(P_t)$ by the first item. We claim that $1-G_t^* G_t|_{R(P_t)}: R(P_t) \to R(P_t)$ is compact if and only if $c^\perp \sin^2 (t(\frac{\pi}{2}c^\perp - Z))$ is compact.  For we observe that $XY=YX$, so that $e^{-tX}e^{tY}=e^{t(Y-X)}$. Therefore, we have that $P_t - G_t^*G_t=e^{tX}[P_0 - P_0 e^{t(Y-X)} Q_0 e^{-t(Y-X)}P_0 ]e^{-tX}$ is compact if and only if $P_0 - P_0 e^{t(Y-X)} Q_0 e^{-t(Y-X)}P_0$ is compact. Set $P_t=t(\frac{\pi}{2}c^\perp-Z)$,  we compute
 \begin{align*}
P_0 e^{t(Y-X)} Q_0 e^{-t(Y-X)} P_0 & = P_0 \begin{pmatrix} \cos(P_t)   &   -\sin(P_t)  \\   \sin(P_t)   &   \cos(P_t) \end{pmatrix}  \begin{pmatrix} 1   &   0  \\   0   &   c\end{pmatrix}
 \begin{pmatrix} \cos(P_t)   &   \sin(P_t)  \\   -\sin(P_t)   &   \cos(P_t) \end{pmatrix}P_0 \\
 & =   \begin{pmatrix} \cos^2(P_t)    + c \sin^2(P_t) &   0  \\   0  &   0 \end{pmatrix}.
\end{align*}
 Therefore, $P_0 - P_0 e^{t(Y-X)} Q_0 e^{-t(Y-X)} P_0 $ is compact if and only if $c^\perp \sin^2(P_t)$ is compact, which proves our claim. Also observe that one can construct  examples, where $Z$ satisfies the previous conditions, such that $X_t:=c^\perp \sin^2(P_t)$ is compact or not. For instance, if $Z$ satisfies $N(Z)^\perp=R(Z)\subseteq N(c)$, then this depends on when $\frac{\pi}{2}c^\perp-Z$ is compact or not.

 \medskip
 
From these observations, we obtain that $(P_t,Q_t) \in \Delta$ for all  $t \in (-\frac{1}{2},\frac{1}{2})$, $t \neq 0$,   if and only if we take $Z$ such that $X_t$ is not compact. Thus, we can choose $Z$ satisfying this condition to get a family of curves $\delta^t(s)=(\delta^t_1(s), \delta^t_2(s))$, $s \in [0,1]$, defined by 
\begin{align*}
\delta^t_1(s)& :=e^{2stX}e^{-tX} P_0 e^{tX} e^{-2stX},\\
\delta^t_2(s) &:= e^{2stY}e^{-tY} Q_0 e^{tY} e^{-2stY},
\end{align*} 
which are geodesics in $\p(\h) \times \p(\h)$ joining $(P_{-t},Q_{-t}) \in \fde$ and $(P_t,Q_t)\in \fde$ such that $\delta^t(s)\in \fde$, $s\neq \frac{1}{2}$ and $\delta^t(\frac{1}{2})=(P_0,Q_0) \in \fga$. Furthermore, we know that $\delta^t$ has minimal length since $t\in (-\frac{1}{2},\frac{1}{2})$, $\|Z\| \leq \frac{\pi}{2}$,  $X$ is co-diagonal with respect to $P_0$ and $Y$ is co-diagonal with respect to $Q_0$.  Minimizing geodesics of $\p(\h)$ are uniquely determined for endpoints lying at distance less than $1$. Hence, the points $(P_{-t,}Q_{-t}), (P_t,Q_t) \in \fde$ can be made arbitrary close to each other if one takes $t$ small enough, and they  cannot be joined by a minimizing geodesic contained in $\fde$. 
 \end{ejem}

Properties of the Fredholm index can be used to derive sufficient conditions to obtain minimizing geodesics contained in $\fde$ or $\fga$ that join pairs in $\p_{\infty,\infty} \times \p_{\infty,\infty}$. 
We now recall the definition and some properties of the restricted Grassmannian; see for instance \cite{carey, PS, SW}. The restricted Grassmannian is given by 
 \begin{equation}\label{restricted grass}
 Gr_{res}= Gr_{res}(\h_+,\h_-):=\{  P \in \p(\h) : P - P_+ \text{ is compact}  \},
 \end{equation}
 for some fixed projections $P_{\pm}$ associated to a decomposition $\h_+ \oplus \h_-=\h$, where $\h_{\pm}$
 are both infinite dimensional subspaces. 
Recall from Remark \ref{remark para Gamma1} that given two projections $P,Q$ such that $P-Q$ is compact, then $Q|_{R(P)}:R(P) \to R(Q)$ is a Fredholm operator, with index 
 \begin{align*}
\hbox{index}(P,Q) &:=\text{index}(QP|_{R(P)}:R(P) \to R(Q))\\
 & =\dim(R(P)\cap N(Q))- \dim(R(Q) \cap N(P)).
 \end{align*}
 We observe that the sets 
 $$
 Gr_{res}^{(k)}:=\{  P \in Gr_{res} :  \hbox{index}(P_+,P)=k  \}, \,\,\,\, k \in \Z,
 $$  
 describe all the connected components of $Gr_{res}$.  
 
\begin{prop}\label{prop index rest}
The following assertions hold:
\begin{itemize}
\item[i)]   For every $k \in \Z$, $Gr_{res}^{(k)} \times Gr_{res}^{(k)} \subseteq \fde$, and there is a minimizing geodesic contained in $\fde$  joining any pair of points in $Gr_{res}^{(k)} \times Gr_{res}^{(k)}$. 
\item[ii)]  For $k\neq j$,  $Gr_{res}^{(k)} \times Gr_{res}^{(j)} \subseteq \fga$, and there is a minimizing geodesic contained in $\fga$  joining any pair of points in $Gr_{res}^{(k)} \times Gr_{res}^{(j)}$. 
\end{itemize}
\end{prop}
 \begin{proof}
$i)$  Given three projections $P_1, P_2, P_3$  such that $(P_1, P_2)$ and $(P_2, P_3)$ are Fredholm pairs, and either $P_1 - P_2$ is compact or $P_2 - P_3$ is compact, then $(P_1,P_3)$ is a Fredholm pair and $\hbox{index}(P_1, P_3)=\hbox{index}(P_1, P_2) + \hbox{index}(P_2 , P_3)$ (see \cite[Thm. 3.4]{ass}). Thus, if one has two projections $P,Q \in Gr_{res}$ satisfying $\hbox{index}(P_+,P)=\hbox{index}(P_+,Q)=k$, then $\hbox{index}(P,Q)=\hbox{index}(P,P_+) + \hbox{index}(P_+, Q)=-\hbox{index}(P_+,P) + \hbox{index}(P_+, Q)=0$. That is, we obtain that $(P,Q) \in \fde$.  Take now two pairs $(P_0,Q_0), (P_1,Q_1) \in Gr_{res}^{(k)} \times Gr_{res}^{(k)}$. By the same properties of the index, we find that $\hbox{index}(P_0,P_1)=\hbox{index}(Q_0,Q_1)=0$. It was shown in \cite{AL08} that there exists a minimizing  geodesic joining every pair of points in the same connected of the restricted Grassmannian.  
Actually this results was proved for the  restricted Grassmannian relative to the ideal of Hilbert-Schmidt operators; but it is not difficult to check that the same holds for the above defined restricted Grassmannian relative to the ideal of compact operators. Thus, there is a minimizing geodesic $\delta(t)=(P_t,Q_t) \in \fde$, $t \in [0,1]$, joining the points $(P_0,Q_0)$ and $(P_1, Q_1)$. 

\medskip

\noi $ii)$ Similarly as in the previous item, we see that $\hbox{index}(P,Q)\neq 0$, whenever $\hbox{index}(P_+,P)=k$ and $\hbox{index}(P_+,Q)=j$ ($k\neq j$). Also note that $P - PQP=P(P-Q)P=P(P-P_+)P + P(P_+ -Q)P$ is compact since we assume that $P,Q \in Gr_{res}$. Hence $(P,Q)\in \fga$. Again the same argument as in the previous item shows that there is a minimizing geodesic $\delta(t)=(P_t,Q_t) \in \fga$, $t \in [0,1]$, joining any pairs of points $(P_0,Q_0)$ and $(P_1, Q_1)$ with $(P_0,Q_0), (P_1, Q_1) \in Gr_{res}^{(k)} \times Gr_{res}^{(j)} $. 
\end{proof}

\section{Examples in function spaces}\label{examples section}

 In  previous works \cite{AC19, grassH2, ACV21} we studied the condition for the existence of minimal geodesics joining two subspaces in certain functional spaces. In this section we discuss if those subspaces and other related subspaces satisfy the weaker condition of having a common complement.

 Let $L^2$ and $L^\infty$ be the spaces of square-integrable  and essentially bounded complex valued functions with respect to the Lebesgue measure on   the unit circle $\mathbb{T}=\{ z \in \CC : |z|=1 \}$. 
For  $p=2,\infty$, we consider the corresponding Hardy spaces
 $$
H^p =   \{f\in L^p: \hat{f}(n)=0 \hbox{ for } n<0\}.
$$
Here $\{ \hat{f}(n) \}_{n \in \Z}$ are the Fourier coefficients of  $f \in L^2$. These are closed subspaces of $L^p$.  As usual we identify $H^2$
with the space $H^2(\D)$ consisting of all the functions $f$ analytic in the unit disk $\D$ satisfying 
$$
\| f \|_{H^2(\D)}:=\left(\sup_{0 \leq r<1} \frac{1}{2\pi} \int_0 ^{2\pi} |f(r e^{it})|^2 dt \right)^{1/2} < \infty.
$$
Given  $f \in H^2(\D)$, one has that $\lim_{r \to 1^-} f(r e^{it})=f^*(e^{it})$ exists for the $L^2$ norm, and $\| f^*\|_{L^2}=\| f \|_{H^2(\D)}$. Analogously the subspace $H^\infty$ is identified with the space $H^\infty(\D)$  consisting of bounded analytic functions on $\D$ endowed with the norm $\|f \|_{H^\infty(\D)}=\sup_{z \in \D}|f(z)|$; though in this case the  limit defining $f^*$ has to be taken in the weak* - topology.  These    classical results and others mentioned without references in this section can be found for instance in the  books    \cite{BS,  douglas, nikolski}.

Our first set of  examples concerns subspaces of the form $f H^2$, $f \in L^\infty$. These are  closed subspaces of $L^2$ if and only if $f \in (L^\infty)^\times$, where $(L^\infty)^\times$ are the invertible functions in the algebra $L^\infty$ (see \cite[Lemma 3.1]{grassH2}). 
In such case, $f H^2$ is a (single) shift-invariant subspace. Conversely, the Beurling-Helson theorem gives that 
a closed (single) shift-invariant subspace must be of the form $f H^2$, for some $f \in L^\infty$, $|f|=1$ a.e. 
 We consider the Riesz projection 
 $$
 P_+:L^2 \to H^2, \, \, \, \, P_+\left(\sum_{n \in \Z} \hat{f}(n) e^{int}\right)=\sum_{n \geq 0} \hat{f}(n) e^{int}.
 $$
We recall that the Toeplitz operator $T_f:H^2 \to H^2$ with symbol $f\in L^\infty$ is defined by $T_f h=P_+(fh)$, 
for $h \in H^2$.

\begin{ejem}[\textbf{Sarason algebra and  restricted Grassmannian  of the Hardy space}]\label{Grassm restringida}
Let $C$ denote the space of continuous complex valued functions on $\mathbb{T}$. The \textit{Sarason algebra} is defined as 
$$
H^\infty + C=\{  f + g  :  f \in H^\infty, \, g \in C \}.
$$
It is the smallest closed subalgebra of $L^\infty$ that properly contains $H^\infty$. We write $(H^\infty + C)^\times$ for the invertible functions of the algebra $H^\infty + C$, which has a characterization in terms of the harmonic extension. Recall that the harmonic extension to the unit disk of a function $f \in L^\infty$ is given by
$$
(hf)(r e^{it})=\frac{1}{2\pi}\int_0^{2\pi} k_r(t-s)f(e^{is}) ds=  \sum_{n \in \Z} \hat{f}(n) r^{|n|} e^{int}, \, \, \, \, 0 \leq r <1, \, \, \, 0 \leq t < 2 \pi ,
$$
where $k_r(t)=(1-r^2)/(1 -2r \cos(t) + r^2)$  is the Poisson kernel.
   Then,  one has that $f \in (H^\infty + C)^\times$  if and only if there exist $\delta, \epsilon >0$ such that $|(hf)(re^{it})|\geq \epsilon$ for 
$1-\delta < r <1$. For $f \in (H^\infty + C)^\times$, one can define an index $ind(f)$ as minus the winding number with respect to the origin of the curve $(hf)(re^{it})$ for $1-\delta <r <1$. This index is stable under small perturbations and it is an homomorphism of  $(H^\infty + C)^\times$ onto the group of integers. 

In this example, we take $L^2$ as the underlying Hilbert space, and  subspaces of the form 
\begin{equation*}
\s=f H^2, \, \, \, \, \, \t=g H^2, \, \, \, \, \, f, g \in (H^\infty + C)^\times .
\end{equation*}  
Let $Gr_{res}$ the corresponding restricted Grassmannian associated with the Riesz projection $P_+$ (see \eqref{restricted grass}).  According to \cite{grassH2}, for $f$ an invertible function in $L^\infty$, we have  $f \in (H^\infty+ C)^\times$ if and only if $P_{f  H^2} \in Gr_{res}^{(k)}$, where $k=-ind(f)$. Applying Proposition \ref{prop index rest} we obtain that the subspaces $f H^2$, $g H^2$ admit a common complement  if and only if $ind(f)=ind(g)$. 
Notice that $1 - G^*G|_{N(G)^\perp}$ is always compact since $P_\s-P_+$ and $P_\t - P_+$ are compact operators for $P_\s, P_\t \in Gr_{res}$. In the case in which $ind(f)\neq ind(g)$, we observe that $(P_{f H^2}, P_{g H^2}) \in \fga_1^n$, where $n=ind(g)-ind(g)$.
\end{ejem}

\begin{rem} We observe some relevant cases contained in the previous example, and we recall several ways to compute the index, which can be useful to construct concrete examples.

\medskip

\noi $i)$ Let $(H^\infty)^\times$ and $C^\times$ be the invertible functions in $H^\infty$ and $C$, respectively. 
Of course,  $C^\times$ are the non-vanishing continuous functions on $\mathbb{T}$; meanwhile,  for $f \in H^\infty$, one has $f \in (H^\infty)^\times$ if and only if   $|(hf)(z)|\geq \epsilon >0$, for all $z \in \D$. 
Since $(H^\infty)^\times \subset (H^\infty + C)^\times$ and  $C^\times \subset (H^\infty + C)^\times$, we can give as particular examples subspaces of the form $fH^2$, where $f \in (H^\infty)^\times$ or $f \in C^\times$.  

At this point, it can be of interest to recall another characterization of the  invertible functions  in the  Sarason algebra. For $f \in H^\infty + C$, one has that $f \in   (H^\infty + C)^\times$
if and only if $f=f_1 f_2$, for some functions $f_1\in (H^\infty)^\times$ and $f_2 \in C^\times$ (see \cite[Section 2.5.3]{nikolski}).

\medskip

\noi $ii)$ The index of a non-vanishing differentiable function $f$ on $\mathbb{T}$ is given by 
$$
ind(f)=\frac{1}{2\pi} \int_0^{2\pi} \frac{f'(e^{it})}{f(e^{it})}e^{it} dt.
$$
When $f$ is $C^1$ and $|f(e^{it})|=1$, $0 \leq t < 2 \pi$, it can be rewritten as 
$$
ind(f)=\sum_{n \in \Z} n |\hat{f}(n)|^2.
$$
In the case of a rational function $f$ which does not vanish on $\mathbb{T}$, $ind(f)=z-p$, where $z$ and $p$ are the number of zeros and poles of $f$ in $\D$ counted with multiplicities.
\end{rem}

We now recall the inner-outer factorization. A function $f \in H^2$ is called \textit{inner} if $|f(e^{it})|=1$ a.e. on $\mathbb{T}$. A function $f \in H^2$ is \textit{outer} if $\overline{\mathrm{span}}\{f \chi_n : n \geq 0 \}=H^2$, where $\chi_n$ are the functions on $\T$ defined by  $\chi_n(e^{it})=e^{int}$ ($n \in \Z$). Every $f \in H^2$ can be expressed as  $f=f_{inn} f_{out}$, where $f_{inn}$ is an inner function and $f_{out}$ is an outer function. This factorization is unique up to a multiplicative constant. 

An inner function $f$ can  also be factorized uniquely as $f=\lambda b s$, where $\lambda \in \T$, $b$ is a Blaschke product  and $s$ is a singular inner function. This is known as the \textit{canonical factorization}. We recall that a \textit{Blaschke product} is an inner function $b$ with zeros $\{ a_j\}_1^n $ in $\D$ ($1 \leq n \leq \infty$) defined by
$$
b(z)=\prod_{j=1}^n \frac{\bar{a}_j}{|a_j|}\frac{a_j-z}{1-\bar{a}_j z}, \, \, \, \, z \in \D.
$$
For $a_j=0$ we interpret $\bar{a}_j/|a_j|=-1$.  
In the case where $n=\infty$, the sequence of zeros $\{  a_j \}_{j\geq 1}$ satisfies the Blaschke condition $\sum_{j \geq 1} 1 - |a_j|<\infty$, and  the product converges uniformly on compact subsets of $\D$. 
 On the other hand, a \textit{singular inner function} $s=s_\mu$ is defined by
 $$
 s(z)=\exp \left( -\int_\T \frac{\psi + z}{\psi -z}  d \mu(\psi) \right), \, \, \, z \in \D,
 $$
for $\mu$ a positive finite measure on $\T$ that is singular with respect to the Lebesgue measure. 
Thus, for $f$ an inner function, $f=\lambda b s$, $b$ is the  Blaschke product is associated with the zero set of $f$  and $s > 0$ on $\D$. 

\begin{ejem}[\textbf{Inner functions with distinct supports}]\label{inner distinct}
The \textit{support} of an inner function $f$ denoted by $supp(f)$ is the set of all points in $\T$ that are limits of zeros of $f$ or belong to the  support of $\mu$ (here $s=s_\mu$). In this example, we consider $\h=L^2$, $\s=f H^2$ and $\t=g H^2$, for $f$, $g$ inner functions with distinct support. Thus, there is point $z_0 \in supp(f)\setminus supp(g)$ or  $z_0 \in supp(g)\setminus supp(f)$.  We may assume that the first of these conditions holds true for our purpose  of determining if $fH^2$ and $gH^2$ have a common complement. Then from the proof of \cite[Thm. 1]{SL71} we know that $T_{f\bar {g}}$ is not left invertible in $H^2$, which is equivalent to say that $T_{f \bar{g}}$ is not injective or $R(T_{f\bar{g}})$ is not closed. Below we consider these two cases.

 In the first case,  $T_{f \bar{g}}$ is not injective, so by Coburn's lemma we know that $T_{f\bar{g}}^*=T_{ \bar{f}g}$ is  injective. Hence $\dim fH^2 \cap (gH^2)^\perp = \dim N(T_{f \bar{g}}) > 0 $ and $\dim gH^2 \cap (fH^2)^\perp =\dim N(T_{g\bar{f}})=0$. 
 Then, we have to discuss   whether $1- G^*G|_{N(G)^\perp}$ is compact or not, where we take the operator $G$ as  $G:=P_\s |_{\t}: \t \to \s$. 
Note that $P_\s=M_f P_+ M_{\bar{f}}$ and $P_\t=M_g P_+ M_{\bar{g}}$. Here $M_h$ is the multiplication operator  by a function $h \in L^\infty$ acting on $L^2$.  
Then 
$$
G^*G=P_\t P_\s P_\t|_\t= M_g P_+ M_{ f\bar{g} } P_+ M_{g \bar{f}} P_+ M_{\bar{g}}|_{gH^2},
$$
and
$$
N(G)=N(G^*G)=M_g N(P_+ M_{f \bar{g}} P_+ M_{g\bar{f}}P_+).
$$
The unitary operator $M_{\bar{g}}=M_g^*$ maps $\t$ onto $H^2$.  Note also that $M_{\bar{g}}$ maps $N(G)$ onto 
$$
N(P_+ M_{f\bar{g}} P_+ M_{g\bar{f} }P_+)=N(P_+M_{g\bar{f}}P_+).
$$
Then instead of showing that $1-G^*G\big|_{N(G)^\perp}$ is compact, we can show, equivalently, that 
$$
1-P_+ M_{f\bar{g}} P_+ M_{g\bar{f}}P_+\big|_{N(P_+M_{g\bar{f}}P_+)^\perp}
$$
is compact. Note that  $P_+M_{g \bar{f}}P_+$ on $H^2$ coincides with the Toeplitz operator $T_{g\bar{f}}$.
Since  we have assumed that $N(T_{ \bar{f}g})=\{ 0 \}$,  our task now reduces to analyze if the operator
\begin{equation}\label{semicomm}
1-T_{\bar{f}g}^*T_{\bar{f}g}
\end{equation}
is compact in $H^2$. To this end, we need to recall here some facts.  The operator $T_v T_u - T_{vu}$ is called the semicommutator of  the Toeplitz operators $T_u$ and $T_v$ for $u,v \in L^\infty$. The Hankel operator with symbol $u \in L^\infty$ is the bounded linear operator $H_u:H^2 \to H^2_-$ defined by $H_u(f)=P_-(uf)$, where $f \in H^2$ and $P_-$ is the orthogonal projection onto $H^2_-:=(H^2)^\perp$. The following well-known formula relates semicommutators and Hankel operators: 
$$
T_{uv} - T_u T_v=H_{\bar{u}}^* H_v
$$  
Recall also Hartman's theorem, which states that $H_u$ is compact if and only $u \in H^\infty + C$.  
Applying these results with $u=f\bar{g}$, $v=\bar{f} g$, we find that the operator in \eqref{semicomm} is compact if and only if $H_{\bar{f}g}^*H_{\bar{f}g}$ is compact.  This is equivalent to say that $H_{\bar{f}g}$ is compact, or by Hartman's theorem, that $\bar{f}g \in H^\infty + C$. 

We now show that this is a contradiction with our assumption on the supports of $f$ and $g$. Indeed, for $z_0 \in supp(f)$ there is a sequence $\{ z_n \}_{n \geq 1}$ in $\D$ such that  
$z_n \to z_0$ and $(hf)(z_n) \to 0$ (\cite[Lemma 2]{SL71}).  On the other hand, if $z_0 \notin supp(g)$, then $|(hg)(z_n)|\to 1$. Notice that the harmonic extensions $hf$ and $hg$ are the analytic extensions because $f$, $g$ are inner. Take $a \in H^\infty$, $b \in C$ and $\epsilon >0$. Recall that the harmonic extension is asymptotically multiplicative on $H^\infty + C$, so there exists a compact set $K \subset \D$ such that $\| (hf(a+b))(z) - (hf)(z)\, (ha)(z) - (h f)(z)\, (hb)(z)\|  < \epsilon$, for all $z \in \D \setminus K$. Then, we have
\begin{align*}
\| \bar{f} g - (a+b) \|_{L^\infty(\mathbb{T})} & = \|  g - f(a+b) \|_{L^\infty(\mathbb{T})}  \\ 
&  \geq \|   hg - hf(a+b) \|_{L^\infty(\D)}  \geq \|   hg - hf(a+b) \|_{L^\infty(\D\setminus K)}  \\
& \geq \| hg  - (hf) (ha) - (h f) (hb)\|_{L^\infty(\D \setminus K)} - \\ & - \| hf(a+b) - (hf) (ha) - (h f) (hb)\|_{L^\infty(\D \setminus K)} \\
& \geq 1- \epsilon.
\end{align*}
We have used that $|(hg)(z_n)| \to 1$, $(hf)(z_n)(ha)(z_n) \to 0$ since $ha$ is bounded for $a \in H^\infty$ and $ (h f)(z_n) (hb)(z_n) \to 0$ since $hb$ is continuous on $\overline{\D}$ for $b \in C$. The above arguments imply 
$d(\bar{f}g, H^\infty + C)=1$, and in particular, we obtain $\bar{f}g \notin H^\infty +C$. Hence the operator  $1-G^*G\big|_{N(G)^\perp}$  is not compact.

In the second case, we suppose that $T_{f \bar{g}}$ is injective and $R(T_{f\bar{g}})$ is not closed. We change the definition of $G$ here by setting   $G:=P_\t |_{\s}: \s \to \t$. Similar arguments as in the previous case show that, under the assumption $N(T_{f \bar{g}})=\{ 0 \}$, we have the operator $1- G^*G|_{N(G)^\perp}$ is compact if and only if $1 - T_{\bar{g}f}^*T_{\bar{g}f}$ is compact in $H^2$. But $R(T_{f\bar{g}})$ is not closed, and consequently,  neither $R(|T_{f\bar{g}}|^2)$ is  closed, so that $0$ is not an isolated point of its spectrum. Hence $1$ is not an isolated point of the spectrum of $1 - T_{\bar{g}f}^*T_{\bar{g}f}$, which clearly implies that this operator cannot be compact. 

From the above two cases, we conclude that $fH^2$ and $g H^2$ have a common complement, whenever $f$ and $g$ are inner functions with distinct support. 
\end{ejem}

\begin{rem}
\noi $i)$ The pair of subspaces considered in Example \ref{Grassm restringida} differ from that of Example \ref{inner distinct}. Indeed, if $fH^2$ and $gH^2$ are subspaces such that $f$, $g$ inner functions with $supp(f)\neq supp(g)$,  
then we must have that  $supp(f) \neq \emptyset$ or $supp(g)\neq \emptyset$. But the unique inner functions invertible in $H^\infty + C$ are the finite Blaschke products (see \cite[Sec. 2.5.3]{nikolski}). 

\medskip

\noi $ii)$ We observe that under the stronger assumption $supp(f)\setminus supp(g)\neq \emptyset$ and $supp(g)\setminus supp(f)\neq \emptyset$, it can be shown that $N(T_{\bar{f}g})\simeq gH^2 \cap (fH^2)^\perp=\{ 0 \}$ and $N(T_{f\bar{g}})\simeq fH^2 \cap (gH^2)^\perp=\{ 0 \}$ (see \cite[Thm. 1-2]{SL71} and \cite[Ex. 5.6]{grassH2}). However, one can give examples of inner functions $f$, $g$ with $supp(g)\subset supp(f)$ and $N(T_{\bar{f}g})\neq \{ 0 \}$. As a simple example, take $h$ a Blaschke product with infinite zeros converging to  $z_0 \in \T$,  $g$ another Blaschke product with infinite zeros converging to $z_1 \in \T$, $z_0 \neq z_1$ and take $f=g h$. Clearly, we have $N(T_{\bar{f}g})=N(T_{\bar{h}})\neq \{ 0 \}$. 
\end{rem}

\begin{ejem}[\textbf{Inner functions with same supports}]
We take $\h=L^2$, $\s=fH^2$ and $gH^2$ for $f$, $g$ inner functions with the same support. Furthermore, we consider only the case $supp(f)=supp(g)=\{ 1 \}$. 
This situation is illustrated by the following examples. 

For the first type of examples we consider $f$ a singular inner function and $g$ a Blaschke product. More precisely, let $f_a(z)=\exp(a(z+1)/(z-1)) $, $a > 0$. A Koosis function is a Blaschke product $g$ such that $f_a \bar{g} \in H^\infty + C$ for all $a>0$.
A Koosis function is characterized in terms of the sequence of zeros $\{ a_k \}_{k \geq 1}$ of the Blaschke product $g$. Indeed, let $\lambda_k=i (1+ a_k)/(1-a_k)$, $k \geq 1$, which defines a sequence in the  upper-half plane. Then,  $g$ is a Koosis function if and only if ${\rm Im}(\lambda_k) \to \infty$ and 
$$
\lim_{|x| \to \infty} \sum_k \frac{{\rm Im}(\lambda_k)}{|x-\lambda_k|^2}=0, \, \, \, \, x \in \RR.
$$ 
 In particular, $g$ is a Koosis function if it has real zeros converging to $1$. 
It was proved that $N(T_{\bar{f}_a g})\simeq gH^2 \cap (f_a H^2)^\perp$  is infinite dimensional and $R(T_{\bar{f}_a g})=L^2$  for every Koosis function $g$ and $a >0$ (see \cite[Thm. 4]{SL71}). Thus, we have $\dim  gH^2 \cap (f_a H^2)^\perp=\infty $ and $\dim  f_aH^2 \cap (g H^2)^\perp=0$. We can repeat the argument of Example \ref{inner distinct}, where now $N(T_{\bar{f}_a g}) =\{ 0 \}$ and $G=P_{gH^2}|_{f_aH^2}:f_a H^2 \to g H^2$, to find that the operator $1-G^*G|_{N(G)^\perp}$ must be compact because $f_a \bar{g} \in H^\infty + C$ by definition of a Koosis function. Hence $f_a H^2$ and $g H^2$ do not admit a common complement, and $(P_{f_a H^2}, P_{gH^2}) \in \fga^r_\infty$.

The second type of examples concerns two infinite Blaschke products converging to $1$. It was mentioned in \cite{GS81} that there exist  infinite Blaschke products $f$, $g$ such that $f/g \in (H^\infty + C)^\times$ (i.e. $f$, $g$ are codivisible in $H^\infty + C$). This can be constructed by taking two infinite Blaschke products with zeros converging to one point in $\T$, where one of the zeros is a suitable perturbation of the other. Using this idea for studying examples of geodesics in the Grassmann manifold, we proved in \cite[Thm. 5.5]{ACV21} that for every integer $n\geq 0$ there exist two disjoint sequences $\{ a_k\}_{k \geq 1}$ and $\{ b_k \}_{k \geq 1}$ in $\D$ satisfying Blaschke condition, $a_k \to 1$, $b_k \to 1$, and their corresponding Blaschke products $f$ and $g$ are such that $ \dim (fH^2)^\perp  \cap gH^2=0$ and $\dim fH^2 \cap (gH^2)^\perp=n$. Furthermore, $f$ and $g$ are codivisible in $H^\infty + C$, so again by the  previous arguments we have that $1-G^*G|_{N(G)^\perp}$ is compact. Hence there is a common complement for $fH^2$ and $gH^2$ if and only $n=0$. In the case $n \neq 0$, we observe that $(P_{fH^2}, P_{gH^2})\in \fga_1^n$. 
\end{ejem}


\begin{ejem}[\textbf{Functions supported on arcs  and odd functions}]
Let again $\h=L^2=L^2(\mathbb{T})$, let $\j\subset \mathbb{T}$ be the  arc $\j=\{e^{it}: t\in[0,T]\}$, and consider the subspace
$$
\s=\s_\j=\{f\in L^2: \sup( f) \subset \j\}.
$$ 
Here $\sup(f)$ denotes the (essential) support of $f \in L^2$.
We also consider the  subspace
$$
\t=\t_{odd} =\{g\in L^2(\mathbb{T}): \hat{g}(m)=0 \hbox{ for } m \hbox{ even}\}.
$$
Notice that $\s$ and $\t$ are both infinite dimensional subspaces. The following facts follow from straightforward computations.
\begin{itemize}
\item[1.] If $T=\pi$, then $\s_\j$ and $\t_{odd}$ are in generic position, in particular, $\s_\j\cap\t_{odd}^\perp=\{0\}=\s_\j^\perp\cap\t_{odd}$, and therefore $\s_\j$ and $\t_{odd}$ have a common complement.
\item[2.] If $T<\pi$, then $\s_\j\cap\t_{odd}^\perp=\{0\}$. Set $\j_T=\{  e^{it} : -\pi \leq t \leq -T \hbox{ or }    T \leq t \leq \pi \}$, then
$$
\dim \s_\j^\perp\cap\t_{odd}=\dim \{  f \in L^2 : \hat{f}(m)=0  \hbox{ for } m \hbox{ even}, \, \sup(f)\subset \j_T  \}  =\infty.
$$
\end{itemize}
Nevertheless in the second case above $\s_\j$ and $\t_{odd}$ have a common complement. Indeed, note that
$$
1-P_{\s_\j}P_{\t_{odd}}P_{\s_\j}\big|_{\s_\j}=P_{\s_\j} - P_{\s_\j}P_{\t_{odd}}P_{\s_\j}\big|_{L^2(\j)}=P_{\s_\j}(1-P_{\t_{odd}})\big|_{L^2(\j)}.
$$
Next observe that $(1-P_{\t_{odd}})f(e^{it})=\frac12(f(e^{it})+f(-e^{it}))$. If $h\in L^2$ is supported in $\j$, then $h(-e^{it})=0$ a.e.  for $t\in[0,T]$, and one has  that
$$
(1-P_{\s_\j}P_{\t_{odd}}P_{\s_\j})h=\frac12 h,
$$
which clearly implies that $1-P_{\s_\j}P_{\t_{odd}}P_{\s_\j}\big|_{\s_\j}$ is non compact.
\end{ejem}

 

The Hardy space $H^2$ is  a reproducing kernel Hilbert space,  where the reproducing kernels are given by the Szego kernels
$$
k_b(z)=\frac{1}{1- \bar{b}z}, \, \, \, z,b \in \mathbb{D}.
$$
For our next example we recall that that a sequence $\{ b_j \}_{j \geq 1}$ in $\D$ is said to be an \textit{interpolating sequence} if there exists  $\delta>0$  such that 
\begin{equation}\label{condicion C}
 \delta < \prod_{\substack{j\ge 1 \\ j\ne i}}\Big|\frac{b_i-b_j}{1-\bar{b}_jb_i}\Big| \ , \  \ i=1,2,\dots.
 \end{equation}
 Let $\Pi_\BB$ denote the Blaschke product associated to a sequence $\BB=\{ b_j \}_{j \geq 1}$ satisfying   Blaschke condition.  Carleson proved the following result: the  normalized Szego kernels   form a Riesz basis $\{  \frac{k_{b_i}}{\| k_{b_i} \|} \}_{i\geq 1}$ of the model space $\k_{\Pi_\BB}=H^2 \ominus \Pi_\BB H^2$ if and only $\BB$ is an interpolating sequence.

\begin{ejem}[\textbf{Blaschke products and odd functions}]
Let now $\h=H^2$, and $\BB=\{ b_j \}_{j \geq 1}$ satisfying  the  Blaschke condition. 
Consider the subspaces $\t=\t_{odd}$ of odd functions in $H^2$, and 
$$
\s=\s_\BB=\{ f\in H^2: f(b_j)=0 \hbox{ for } j\ge 1\}.
$$
We give two type of examples in this setting.

In the first case, suppose additionally that $\BB\cap(-\BB)=\emptyset$ and that the (disjoint) union $\BB\cup(-\BB)$ is an interpolating sequence. In what follows we use that $\| k_b \|=(1 - |b|^2 )^{-1/2}$.
 Suppose that $h\in\s_\BB^\perp\cap\t_{odd}=\k_{\Pi_\BB} \cap \t_{odd}$, then 
$h=\sum_{j\ge 1} \gamma_j (1-|b_j|^2)^{1/2}k_{b_j}$ is odd. Since $k_b(-z)=k_{-b}(z)$, we get
$$
\sum_{j\ge 1} \gamma_j (1-|b_j|^2)^{1/2}k_{b_j}(z)=h(z)=-h(-z)=-\sum_{j\ge 1} \gamma_j(1-|b_j|^2)^{1/2}k_{-b_j}(z),
$$
Using that $\BB\cap(-\BB)=\emptyset$, note
$$
0=\sum_{k\ge 1} \delta_k (1-|c_k|^2)^{1/2}k_{c_k}, \hbox{ where } \ \left\{\begin{array}{l} \delta_k=\gamma_j \hbox{ for  } c_k=b_j \, ,\\ \delta_k=\gamma_i \hbox{ for } c_k=-b_i . \end{array} \right.
$$
Since $\{ (1-|c_k|^2)^{1/2}k_{c_k}\}_{k \geq 1}$ form a Riesz basis, it follows that $c_k=0$ for all $k$. Hence $\s_\BB^\perp\cap\t_{odd}=\{0\}$.

Next observe that $\t_{odd}^\perp=\t_{even}$ consists of even function in $H^2$.  Apparently, if $f\in\t_{even}$  and $f(b)=0$, then also $f(-b)=0$. Therefore if $f\in\s_\BB\cap \t_{odd}^\perp$, the products
$$
\frac{b_j-z}{1-\bar{b}_jz}\frac{b_j+z}{1+\bar{b}_jz}=\frac{b^2_j-z^2}{1-\bar{b}^2_jz^2}
$$
divide $f$. Note that these products are even functions. It follows that 
$$
\s_\BB\cap\t_{odd}^\perp=\{g\ \Pi_{\BB\cup -\BB}: g\in H^2 \hbox{ is even}\},
$$
which is an infinite-dimensional subspace. Now take $G:=P_{\s_\BB}P_{\t_{odd}}: \t_{odd} \to \s_\BB$, which has trivial nullspace. Note also that $1- G^*G= 1-P_{\t_{odd}}P_{\s_\BB}P_{\t_{odd}}\big|_{\t_{odd}}=P_{\t_{odd}}P_{\s_\BB}^\perp P_{\t_{odd}}\big|_{\t_{odd}}$.
 Recall that if $u$ is an inner function, then the projection $P_u:L^2 \to \k_u\subset L^2$ is given by
$$
P_uf=f-uP_+(\bar{u}f),
$$
where $P_+$ is the Riesz projection (see  \cite[Prop. 5.14]{garcia ross libro}).

We claim that $P_{\t_{odd}}P_{\s_\BB}^\perp P_{\t_{odd}}\big|_{\t_{odd}}$ is non compact, which would imply that $\s_\BB$ and $\t_{odd}$ have a common complement. 
Clearly $P_{\t_{odd}}P_{\s_\BB}^\perp P_{\t_{odd}}\big|_{\t_{odd}}$ is compact if and only if $P_{\t_{odd}}P_{\s_\BB}^\perp$ is compact in $H^2$. This operator is compact
if and only if its natural extension to $L^2$, namely 
$P_{odd}P_{\Pi_\BB}$ is compact. Here
$P_{odd}f(z)=\frac12(f(z)-f(-z))$ is the orthogonal projection onto odd elements in $L^2$. Indeed, 
$$
P_{odd}P_{\Pi_\BB}=\left\{ \begin{array}{ll} P_{\t_{odd}}P_{\s_\BB}^\perp & \hbox{ in } H^2, \\  0 &\hbox{ in } (H^2)^\perp . \end{array} \right. 
$$
Clearly $P_{odd}P_{\Pi_\BB}$ is compact if and only if $P_{odd}P_{\Pi_\BB}M_{\Pi_\BB}$ is compact, where $M_{\Pi_\BB}$ denotes the (unitary) multiplication operator by the unimodular element $\Pi_\BB$. Note that 
$$
P_{odd}P_{\Pi_\BB}M_{\Pi_\BB}f=P_{odd}(\Pi_\BB f-\Pi_\BB P^+f)=P_{odd} M_{\Pi_\BB}P_-f,
$$
where $P_-$ is the orthogonal projection onto $(H^2)^\perp$. This operator is clearly non compact: consider for instance $g_n=\bar{\Pi}_\BB z^{2n-1}$ for integers $n\le 0$. These elements are orthonormal and belong to $(H^2)^\perp$, and clearly
$$
P_{odd} M_{\Pi_\BB}P_-g_n=z^{2n-1}.
$$

\medskip

In the second case, we consider a variation of the above example. Put $\BB_0$ an interpolating sequence which is symmetric with respect to the origin, i.e.  $\BB_0=\BB\cup-\BB$ for some  $\BB=\{ b_j \}_{j\geq 1} \subseteq \mathbb{D}$. Then,  we have again
$$
\s_{\BB_0}\cap\t_{odd}^\perp=\{g\ \Pi_{\BB_0}: g\in H^2 \hbox{ is even}\}.
$$
But now $\s_{\BB_0}^\perp\cap\t_{odd}$ consists of odd functions of the form $h=\sum_{j\ge 1} \gamma_k (1-|b_k|^2)^{1/2}k_{b_k}+\sum_{j\ge 1} \delta_j(1-|b_j|^2)^{1/2}k_{-b_j}$. Since
  $\{ \frac{ k_{b_j} }{ \|b_j \| } \}_{j \geq 1} \cup \{ \frac{k_{-b_j} }{ \|k_{-b_j} \| }\}_{j\geq 1}$ 
  forms a Riesz basis of $\k_{\BB_0}$, this means that
$$
\gamma_k=-\delta_k, \ k\ge 1.
$$
This clearly defines an infinite dimensional subspace of $H^2$. 
Therefore, here $\s_{\BB_0}\cap\t_{odd}^\perp$ and $\s_{\BB_0}^\perp\cap\t_{odd}$ are infinite dimensional, and hence $\s_{\BB_0}$ and $\t_{odd}$ have a common complement.
\end{ejem}

The following example concerns well-known facts on the Uncertainty Principle in Harmonic Analysis (see for instance \cite{lenard} or the survey article \cite{folland}). We also refer to \cite{AC19} for the relation with the geometry of the Grassmann manifold. 

\begin{ejem}[\textbf{Measurable sets and Fourier-Plancharel transform}]\label{incertidumbre}
Let $I, J\subset\mathbb{R}^n$ be measurable sets with finite and positive Lebesgue measure. Consider $\h=L^2(\mathbb{R}^n)$ with Lebesgue measure and the projections $P_I$ onto the elements of $L^2(\mathbb{R}^n)$ supported in $I$ and $Q_J$ onto the elements whose Fourier-Plancherel transform is supported in $J$. The following facts are known:
\begin{itemize}
\item $R(P_I)\cap R(Q_J)=R(P_I)\cap N(Q_J)=N(P_I)\cap R(Q_J)=\{0\}$ and $
N(P_I)\cap N(Q_J)$ is infinite dimensional (\cite{lenard}).
\item $P_IQ_JP_I$ is compact, in fact, nuclear (\cite{folland}). 
\end{itemize}
Therefore we have the following:
\begin{itemize}
\item[1.] $\s_I=\{f\in L^2(\mathbb{R}^n): \sup(f)\subset I\}$ and $\t_J=\{g\in L^2(\mathbb{R}^n): \sup(\hat{g})\subset J\}$ have a common complement since it is  straightforward to check that  $\s_I\cap\t_J^\perp=\{0\}=\s_I^\perp\cap\t_J$.
\item[2.] $\s_I$ and $\t_J^\perp$ do not have a common complement.
The role of $\t$ is reversed: now $\s_I\cap(\t_J^\perp)^\perp=R(P_I)\cap R(Q_J)=\{0\}$, but $\s_I^\perp\cap\t_J^\perp$ is infinite dimensional. Moreover
$$
1-P_{\s_I}\big|_{\t_J^\perp}=P_I-P_IQ_J^\perp P_I=P_IQ_JP_I
$$
is compact. Hence $(P_{\s_I},P_{\t_J^\perp}) \in \fga_\infty^l$.
\end{itemize}
\end{ejem}



\medskip

\noi 
\textbf{Acknowledgment.}
This work was supported by the grant PICT 2019 04060 (FONCyT - ANPCyT, Argentina), (PIP 2021/2023 11220200103209CO),  ANPCyT (2015 1505/ 2017 0883) and FCE-UNLP (11X974).

\noi E. Andruchow, {\sc  {Instituto Argentino de Matem\'atica, `Alberto P. Calder\'on', CONICET, Saavedra 15 3er. piso,
(1083) Buenos Aires, Argentina }} and {\sc Universidad Nacional de General Sarmiento, J.M. Gutierrez 1150, (1613) Los Polvorines, Argentina}

\noi  e-mail: eandruch@campus.ungs.edu.ar

\medskip

\noi E. Chiumiento, {\sc  {Instituto Argentino de Matem\'atica, `Alberto P. Calder\'on', CONICET, Saavedra 15 3er. piso,
}(1083) Buenos Aires, Argentina}  and {\sc Departamento de Matem\'aica y Centro de Matem\'atica La Plata, Universidad Nacional de La Plata, Calle 50 y 115 (1900) La Plata, Argentina}

\noi  e-mail: eduardo@mate.unlp.edu.ar

\end{document}